\documentclass[11pt]{amsart}
\usepackage{amsmath,amssymb,amsthm}
\usepackage{tikz}

\newcommand{\bbC}{\mathbb{C}}
\newcommand{\Dc}{\mathcal{D}}
\usepackage[pagewise]{lineno}

\newtheorem{theorem}{Theorem}[section] % 1st argument is your name for it
\newtheorem{lemma}[theorem]{Lemma}     % 2nd argument is what is printed

\newtheorem{example}[theorem]{Example}
\newtheorem{remark}[theorem]{Remark}
\newtheorem*{rem*}{Remark}
\newtheorem*{que*}{Question}
\newtheorem*{theorem*}{Theorem}

\title{A Model for Planar Compacta and Rational Julia Sets}
\date{\today}
\author[J. Luo]{Jun Luo}
\address{School of Mathematics\\
    Sun Yat-Sen University\\ Guangzhou 510275, China}
\email{luojun3@mail.sysu.edu.cn}
\author[Y. Yang]{Yi Yang}
\address{School of Mathematics(Zhuhai)\\
Sun Yat-sen University\\
Zhuhai 519082, China}
\email{yangy699@mail.sysu.edu.cn\ (corresponding author)}
\author[X.T. Yao]{Xiao-Ting Yao} \address{School of Mathematics and Statistics, Guangdong University of Technology, Guangzhou 510520, China} \email{yaoxiaoting55@gdut.edu.cn}

\begin{document}

\begin{abstract}
A  {\bf Peano compactum}  means a compact metric space having locally connected components such that at most finitely many of them are of diameter greater than any fixed number $C>0$. Given a compactum $K$ in the extended complex plane $\widehat{\bbC}$, it is known that there is a finest upper semi-continuous decomposition of $K$ into subcontinua such that the resulting quotient space is a Peano compactum. We call this decomposition the {\bf core decomposition} of $K$ with Peano quotient and its elements {\bf atoms of $K$}. We show that for any branched covering $f: \widehat{\bbC}\rightarrow\widehat{\bbC}$ and any atom $d$ of $f^{-1}(K)$ the image $f(d)$ is an atom of $K$. Since  rational functions are branched coverings, our result extends earlier ones that are restricted to more limited cases, requiring that $f$ be a polynomial with $f^{-1}(K)=K$.
%We also determine the atoms for typical planar continua. In particular, we show that every quadratic Julia set $J$ that contains a Cremer fixed point has a unique atom, the whole Julia set $J$.


\end{abstract}
\maketitle

\tableofcontents
\newpage
%\linenumbers

\section{Motivations and Results}

By a {\bf compactum}  we mean a compact metric space. To describe specific aspects of the topology of a compactum $K$, we may analyse certain upper semi-continuous (usc) decompositions  of $K$ into subcontinua. We further require that the resulting quotient spaces satisfy a group of properties (P). Such a partition $\Dc$ as well as its quotient space, still denoted as $\mathcal{D}$, becomes more interesting if it refines every other usc decomposition of $K$ into subcontinua whose quotient space satisfies (P). In such a case,  $\Dc$ is called the {\bf core decomposition of $K$} with respect to (P). See for instance \cite[Definition 1.1]{FitzGerald67} for the core decomposition of a continuum with respect to the property of being  {\bf semi-locally connected}.

In the current paper we study a special type of core decompositions,  when (P) means {\bf being a Peano compactum}.  Here a Peano compactum is a compact metric space whose nondegenerate components are locally connected and form a null sequence, in the sense that for any $C>0$ at most finitely many of them are of diameter greater than $C$.
One may see \cite[Theorem 3]{LLY-2019} for a characterization of Peano compacta lying on the plane.
%Note that a totally disconnected compactum is a Peano compactum and a Peano continuum is a connected Peano compactum.

We will denote by  $\mathfrak{M}^{PC}(K)$ the family of all the usc decompositions $\Dc$ of  $K$ into subcontinua such that the quotient space is a Peano compactum. The member of $\mathfrak{M}^{PC}(K)$ that refines all the others, if it exists, is called the  core decomposition of $K$ with respect to the property of being a Peano compactum, or shortly the {\bf core decomposition of $K$ with Peano quotient}.
When it exists, this core decomposition is denoted by $\Dc_K^{PC}$ and its elements, considered as subcontinua of $K$, are called {\bf atoms of $K$}. We will call the resulting quotient space  the {\bf Peano model of $K$} and the natural projection $\pi: K\rightarrow\Dc_K^{PC}$  the {\bf Peano projection of $K$}. If $K$ is in the extended complex plane $\widehat{\bbC}$, we further set $\widehat{\Dc}_K=\Dc_K^{PC}\cup\{\{z\}: z\notin K\}$ and call $\widehat{\pi}:\widehat{\bbC}\rightarrow\widehat{\Dc}_K$ the {\bf extended Peano projection}. To emphasize $K$, we sometimes write $\pi_K$ (respectively, $\widehat{\pi}_K$) instead of $\pi$ (respectively, $\widehat{\pi}$).
The resulting quotient space $\widehat{\Dc}_K$ is a cactoid \cite[p.76]{Whyburn42}. By \cite[p.171, Theorem (2.2)]{Whyburn42} and \cite[p.175, Theorem (3.3)]{Whyburn42}, we know that a continuum $N$ is a cactoid if and only if there is a a monotone map of $\widehat{\bbC}$ onto $N$.
%Note that even if $\mathcal{D}_K^{PC}$ doe snot exists, it is possible that $\mathfrak{M}^{PC}(K)$ have a member that can not be refined by any of others. Such a decomposition, if it exists, is called an {\bf atomic decomposition} with Peano quotient. When the core decomposition $\mathcal{D}_K^{PC}$ exists  it is the only atomic decomposition.

Recall that a compactum $K$ in $\mathbb{R}$ is always a Peano compactum, thus $\Dc_K^{PC}$ consists of all the singletons $\{x\}$ with $x\in K$. By \cite[Example 7.1]{LLY-2019}, there is a compactum in $\mathbb{R}^3$ that does not have a  core decomposition with Peano quotient. On the other hand, by \cite[Theorem 7]{LLY-2019}, all planar compacta $K$  have a core decomposition with Peano quotient.
When $K$ is a continuum, the core decomposition $\Dc_K^{PC}$ may be different from the one proposed by Moore \cite{Moore25-a}, which partitions an arbitrary continuum into prime parts such that the quotient space is a Peano continuum. See \cite[Example 2.3]{Yang-Yao} for a special continuum in the plane, such that the core decomposition strictly refines the decomposition into prime parts.

Let us focus on compacta $K\subset\widehat{\bbC}$.
If $K$ is {\bf unshielded}, so that $K=\partial U$ for some component $U$ of  $\widehat{\bbC}\setminus K$, the core decomposition $\Dc_K^{PC}$ has been obtained by Blokh, Curry, and Oversteegen in \cite{BCO11,BCO13}.
When $K$ is connected  the Peano model  $\Dc_K^{PC}$ is called the {\bf finest locally connected model} in \cite{BCO11}. When $K$ is disconnected the Peano model $\Dc_K^{PC}$ is called the {\bf finest finitely Suslinian monotone model} in \cite{BCO13}. These models are planar compacta that are still unshielded \cite[Theorem 19]{BCO13}. Here a compactum is {\bf finitely Suslinian} provided that every collection of pairwise disjoint subcontinua whose diameters are bounded away from zero is finite.

Since every locally connected or finitely Suslinian compactum $K\subset\widehat{\bbC}$ is necessarily a Peano compactum \cite[Theorems 1 and 3]{LLY-2019}, the Peano model $\Dc_K^{PC}$  obtained in \cite[Theorem 7]{LLY-2019} includes as special sub-cases the two models given in \cite{BCO11,BCO13}.
The Peano model $\Dc_K^{PC}$ is developed in a way that is independent of the so-called Moore's Theorem \cite[Theorem 22]{Moore25}. It provides a negative answer to  \cite[Question 5.2]{Curry10}:
{\em Is there  a rational function whose Julia set does not have a finest locally connected model?} \ It also provides a solution to the first part of \cite[Question 5.4]{Curry10}:
{\em For what useful topological properties $P$ does there exist a finest decomposition of every Julia set $J(R)$ {\rm(of a rational function $R$)} satisfying $P$ ? Is the decomposition dynamic? Which of these is the appropriate analogue for the finest locally connected model?}

We aim to resolve the latter two parts of \cite[Question 5.4]{Curry10}. To do that, we analyze how the atoms of $K$ are related to those of $f^{-1}(K)$  for any {\bf branched covering} $f:\widehat{\bbC}\rightarrow\widehat{\bbC}$. Note that for a branched covering $f$ there is a finite set $B_f$,  often called the {\bf branch set}, such that $f$ is a local homeomorphism at every $x$ in $\widehat{\bbC}\setminus B_f$. Also note that every rational function $f(x)$ is a branched covering, with $B_f=\{x: f'(x)=0\}$.
Our main result is the following.

\begin{theorem}\label{invariance}
Given a compactum $K\subset\widehat{\bbC}$ and a branched covering $f:\widehat{\bbC}\rightarrow\widehat{\bbC}$, every atom of $L=f^{-1}(K)$ is sent {\bf onto} an atom of $K$. In particular, if $f$ is a rational function and $K$ is the Julia set of $f$ then the extended Peano projection $\widehat{\pi}_K$ semi-conjugates $f:\widehat{\bbC}\rightarrow\widehat{\bbC}$ to the induced map $\widehat{f}: \widehat{\Dc}_K\rightarrow\widehat{\Dc}_K$ that sends every $d\in\widehat{\Dc}_K$ to $f(d)\in\widehat{\Dc}_K$.
\end{theorem}

When $K$ is the Julia set of a polynomial $f$ with no irrationally neutral cycle, Kiwi \cite{Kiwi04} discusses a special case of Theorem \ref{invariance}. Note that the atoms of $K$ are the fibers of $K$ \cite[Definition 2.5]{Kiwi04}. Here two points $x\ne y\in K$ are contained in different fibers if and only if there is a finite set $C$ and a separation $K\setminus C=E_x\cup E_y$ with $x\in E_x$ and $y\in E_y$. By \cite[Theorem 3]{Kiwi04}, if $K$ is even connected then each of its atoms is a union of finitely many prime end impressions of the unbounded Fatou component .

When $K$ is the Julia set of a rational function $f$, the induced map $\widehat{f}$ reduced to $\Dc_K^{PC}$ is either trivial (when $\Dc_K^{PC}$ has a single element) or topologically mixing  (when $\Dc_K^{PC}$ has more than one hence uncountably many  elements). See \cite[Theorem 30]{BCO11} and \cite[Theorem 20]{BCO13} for earlier discussions in more restricted situations, where $K$ is assumed to be unshielded and $f$ is required to be a polynomial satisfying $f^{-1}(K)=K$.

In order to prove  Theorem \ref{invariance}, we shall analyze the closure of a symmetric relation $R_K$, called the {\bf Sch\"onflies relation} on $K$. By \cite[Theorems 5 and 6]{LLY-2019},  $\Dc_K^{PC}$ is determined by $R_K$. See below for more terminology and some basic results from the literature.

Given two disjoint Jordan curves $J_1$ and $J_2$, we will denote by $U(J_1,J_2)$ the component of $\widehat{\bbC}\setminus(J_1\cup J_2)$ that is bounded by $J_1\cup J_2$. This is an annulus  in $\widehat{\bbC}$. Two points $x,y\in K$ are {\em related under $R_K$} provided that either $x=y$ or there exist two disjoint Jordan curves $J_1\ni x$ and $J_2\ni y$  that satisfy two requirements: (1) $\overline{U(J_1,J_2)}\cap K$ has an infinite sequence of components $P_n$ intersecting both $J_1$ and $J_2$; (2) the limit $P_\infty=\lim\limits_{n\rightarrow\infty}P_n$ under the Hausdorff distance contains $\{x,y\}$. See Figure \ref{R_K}, in which $U(J_1,J_2)$ is represented as a circular annulus.
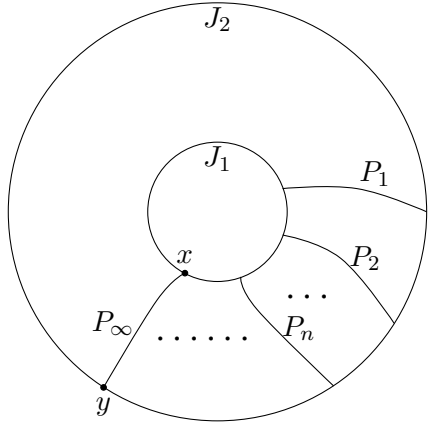
\begin{figure}[ht]
\vspace{-0.382cm}
\begin{center}
\begin{tikzpicture}[scale=0.618]
\draw  (-0.5,3) ellipse (4.5 and 4.5);
\draw  (-0.5,3) ellipse (1.5 and 1.5);
\draw  plot[smooth, tension=.7] coordinates {(0.9,3.5) (2.5,3.5) (4,3)};
\node at (-0.75,0.25) {\Large$\cdots\cdots$};
\node at (1.5,1.15) {\Large$\cdots$};
\draw  plot[smooth, tension=.7] coordinates {(0.9,2.5) (2.15,2) (3.3,0.6)};
\draw  plot[smooth, tension=.7] coordinates {(0,1.6) (0.35,0.9) (2,-0.75)};
\draw  plot[smooth, tension=.7] coordinates {(-1.2,1.68) (-1.8,1.1) (-2.95,-0.78)};

\draw (-1.2,1.68) [fill=black]circle(0.06)node[above]{$x$};
\draw (-2.95,-0.78) [fill=black]circle(0.06)node[below]{$y$};

\node at (-0.5,4.15) {$J_1$};
\node at (-0.5,7.15) {$J_2$};
\node at (2.8893,3.8286) {$P_1$};
\node at (2.6786,2) {$P_2$};
\node at (1.25,0.5) {$P_n$};
\node at (-2.75,0.6) {$P_\infty$};
%\node at (-1.05,1.9) {$x$};
%\node at (-2.5,-0.65) {$y$};
\end{tikzpicture}
\end{center}\vspace{-0.5cm}
\caption{Relative locations of $P_n$ and $P_\infty$ inside the closed annulus $\overline{U(J_1,J_2)}$.}\label{R_K}
\vspace{-0.382cm}
\end{figure}

Denote the {\bf minimal closed equivalence containing $R_K$} by $\sim$ or $\sim_K$, when it is necessary to emphasize $K$.
We also call $\sim_K$ the {\bf Sch\"onflies equivalence on} $K$ \cite[Definition 4]{LLY-2019}.
Note that by \cite[Proposition 5.1]{LLY-2019} every class $[x]_\sim=\{z\in K: \ z\sim x\}$ is a continuum. Therefore, $\Dc_K=\{[x]_\sim: \ x\in K\}$ is an usc decomposition of $K$ into subcontinua. Moreover,  the core decomposition $\Dc_K^{PC}$ exists and coincides with  $\Dc_K$ \cite[Theorem 5 and 6]{LLY-2019}.

Throughout this paper,  $D_r(x)$ denotes the open round disk around $x$ with radius $r$, under the spherical distance $\rho$ on $\widehat{\bbC}$, and $\overline{R_K}$ denotes the closure of $R_K$, as subsets of the product $K\times K$. Clearly, $\overline{R_K}$ is reflexive and symmetric, although  $\overline{R_K}$ is often properly contained in $\sim_K$. In Example \ref{R_K_and_sim}, we give a continuum $K$ such that the fiber $\overline{R_K}[x]=\left\{y: (x,y)\in\overline{R_K}\right\}$ of $\overline{R_K}$ at any $x\in K$ is a proper subset of $K$ while  every fiber of $\sim_K$ equals $K$. In Examples \ref{indecomposable} to \ref{witch_broom}, we give more compacta $K$ that illustrate how $\sim_K$ may be connected to $\overline{R_K}$ in different ways.
\begin{example}\label{R_K_and_sim}
Let $\mathcal{C}$ be the Cantor's ternary set and $\displaystyle K_0=\left\{t+\frac{s}{2}+\frac{\sqrt{3}}{2}s\mathbf{i}: t\in\mathcal{C}, s\in[0,1]\right\}$. Respectively rotating $K_0$ by $120^o$ and $240^o$ about the point $x_0=1$, we will obtain $K_1$ and $K_2$. Set $K=K_0\cup K_1\cup K_2$. See Figure \ref{R_K_and_sim_Fig}. Then every fiber of  $\sim_K$ coincides with $K$. On the other hand, every fiber $\overline{R_K}[x]$ with $x\ne 1$ is a line segment while $\overline{R_K}[1]$ is a Y-shape.
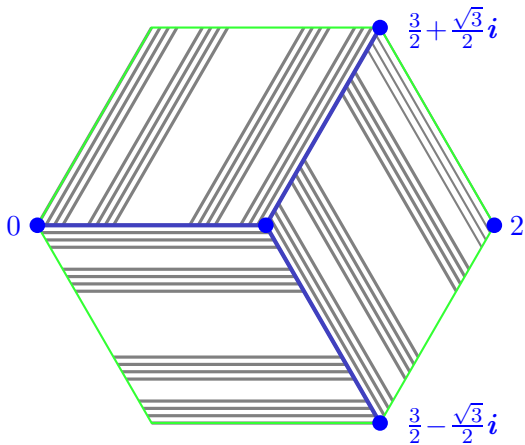
\begin{figure}[ht]
\vspace{-0.382cm}
\begin{tikzpicture}[scale=1.25,x=1cm,y=1cm]
%\draw  (0,0) rectangle (6.25,3);
\foreach \i in {0,...,3}
{
\draw[gray, very thick] (1.215+0.09*\i,2.1044) -- (0.09*\i,0);
\draw[gray, very thick] (1.215+0.54+0.09*\i,2.1044) -- (0.54+0.09*\i,0);
\draw[gray, very thick] (1.215+1.62+0.09*\i,2.1044) -- (1.62+0.09*\i,0);
\draw[gray, very thick] (1.215+2.16+0.09*\i,2.1044) -- (2.16+0.09*\i,0);
\draw[gray, very thick] (2.43+0.045*\i,0-0.078*\i) -- (0.045*\i,-0.078*\i);
\draw[gray, very thick] (2.43+0.27+0.045*\i,0-0.468-0.078*\i) -- (0.27+0.045*\i,-0.468-0.078*\i);
\draw[gray, very thick] (2.43+0.81+0.045*\i,0-1.403-0.078*\i) -- (0.81+0.045*\i,-1.403-0.078*\i);
\draw[gray, very thick] (2.43+1.08+0.045*\i,0-1.871-0.078*\i) -- (1.08+0.045*\i,-1.871-0.078*\i);
\draw[gray, very thick] (2.43+0.045*\i,0.078*\i) -- (3.645+0.045*\i,0.078*\i-2.1044);
\draw[gray, very thick] (2.43+0.27+0.045*\i,0.468+0.078*\i) -- (3.645+0.27+0.045*\i,0.468+0.078*\i-2.1044);
\draw[gray, very thick] (2.43+0.81+0.045*\i,1.403+0.078*\i) -- (3.645+0.81+0.045*\i,1.403+0.078*\i-2.1044);
\draw[gray, thick] (2.43+1.08+0.045*\i,1.871+0.078*\i) -- (3.645+1.08+0.045*\i,1.871+0.078*\i-2.1044);
}

\draw[blue!50!gray, ultra thick] (2.43,0) -- (3.645,-2.1044);
\draw[blue!50!gray, ultra thick] (2.43,0) -- (3.645,2.1044);
\draw[blue!50!gray, ultra thick] (0,0) -- (2.43,0);
\draw[green!80, thick] (0,0) -- (1.215,2.1044)--(3.645,2.1044)--(2.43+2.43,0)--(3.645,-2.1044)--(1.215,-2.1044)--(0,0);
\fill[blue] (2.43,0)  circle (0.5ex); \fill[blue] (4.86,0)  circle (0.5ex);
\fill[blue] (0,0)  circle (0.5ex);
\fill[blue] (3.645,2.10438)  circle (0.5ex);
\node[blue] at (4.4,2.10438) {$\frac32\!+\!\frac{\sqrt{3}}{2}\textbf{i}$};
\node[blue] at (4.4,-2.10438) {$\frac32\!-\!\frac{\sqrt{3}}{2}\textbf{i}$};
\node[blue] at (-0.25,0) {$0$}; \node[blue] at (5.1,0) {$2$};
\fill[blue] (3.645,-2.10438)  circle (0.5ex);
\end{tikzpicture}
\vspace{-0.382cm}
\caption{A simple depiction of $K$ and the fiber $\overline{R_K}[x]$ with $x=1$.}\label{R_K_and_sim_Fig}
\end{figure}
\end{example}

We will obtain the following Theorems \ref{closure} to \ref{invariant-fiber} and use them to prove Theorem \ref{invariance}.

\begin{theorem}\label{closure}
Two points $x\ne y\in K$ are related under $\overline{R_K}$ if and only if $K\setminus(D_r(x)\cup D_r(y))$ has infinitely many components intersecting both $\partial D_r(x)$ and $\partial D_r(y)$, for all $r\in(0,\frac12 \rho(x,y))$.
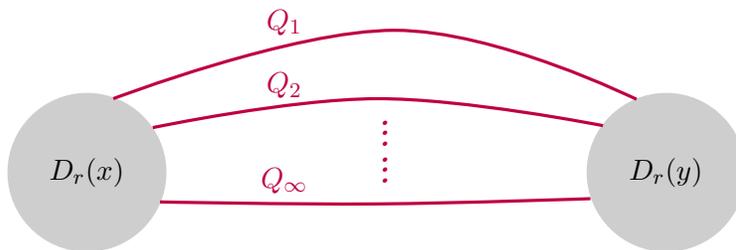
\begin{figure}[ht]
\vskip -0.5cm
\begin{center}
\begin{tikzpicture}[scale=0.7]
\draw[fill,gray!38.2]  (-7.5,1.5) circle (1.5);
%\draw  [fill](-7.5,1.5)circle (0.05);
\draw[fill,gray!38.2]  (3.5,1.5) circle (1.5);
%\draw  [fill](3.5,1.5)circle (0.05);

\draw[very thick, purple]  plot[smooth, tension=.7] coordinates {(-7,2.9) (-1.65,4.2) (2.9375,2.8875)};
\draw[very thick, purple]  plot[smooth, tension=.7] coordinates {(-6.25,2.35) (-2,2.9) (2.3,2.3875)};
\draw[very thick, purple]  plot[smooth, tension=.7] coordinates {(-6.1125,0.9375) (-2,0.9) (2.0625,1)};

\node[very thick, purple] at (-3.75,4.35) {$Q_1$};
\node[very thick, purple] at (-3.75,3.15) {$Q_2$};
\node[very thick, purple] at (-3.75,1.35) {$Q_\infty$};

\node at (-7.5,1.5) {$D_r(x)$};
\node at (3.5,1.5) {$D_r(y)$};

\node[very thick, purple] at (-1.85,2.4) {\LARGE$\vdots$};
\node[very thick, purple] at (-1.85,1.7) {\LARGE$\vdots$};

\end{tikzpicture}
\vskip -0.5cm
\caption{A depiction of the components $Q_i$ of $K\setminus(D_r(x)\cup D_r(y))$ in Theorem \ref{closure}.}
\end{center}
\vskip -0.25cm
\end{figure}
\end{theorem}

\begin{theorem}\label{connected-fiber}
For any $x\in K$ the fiber  $\overline{R_K}[x]=\left\{y: (x,y)\in\overline{R_K}\right\}$ is connected.
\end{theorem}

\begin{theorem}\label{invariant-fiber}
Given a compactum $K\subset\widehat{\bbC}$ and a branched covering $f:\widehat{\bbC}\rightarrow \widehat{\bbC}$, if $u$ is a point in the preimage $L=f^{-1}(K)$ and $x=f(u)$ then $f\left(\overline{R_L}[u]\right)=\overline{R_K}[x]$.
\end{theorem}

\begin{remark}
It remains unknown  whether $R_K$ itself is already a closed relation.  Theorem \ref{closure} provides a nontrivial characterization of the closure $\overline{R_K}$. By Theorem \ref{fiber-and-imp}, if $D\subset\widehat{\bbC}$ is a simply connected region and $K=\widehat{\bbC}\setminus D$ contains more than one point then for any principal point $x$ of a prime end of $D$, the impression of this prime end is contained in $\overline{R_K}[x]$. When $x$ is not a principal point, the previous containment may not hold. See Example \ref{fiber-and-imp-0}.
\end{remark}
%On the other hand, by Theorem \ref{indecomposable_criterion} we know that a full continuum $K\subset\widehat{\bbC}$ is indecomposable if and only if $\overline{R_K}[x]=K$ for all $x\in K$.

\begin{remark}Note that in Theorem \ref{invariance} the preimage $f^{-1}(d)$ of any $d\in\Dc_K^{PC}$ has finitely many components, each of which is an element of $\Dc_L^{PC}$. We call this special phenomenon the {\bf invariance of atoms} under branched coverings. On the other hand, in Theorem \ref{invariant-fiber} it is possible that $f^{-1}\left(\overline{R_K}[x]\right)$ is a continuum that properly contains $\overline{R_L}[u]$ for all $u\in f^{-1}\left(\overline{R_K}[x]\right)$. See Example \ref{inverse_of_fibers}.
\end{remark}

%we will give a continuum $K\subset\widehat{\bbC}$ and a point $x\in K$ such that the preimage $f^{-1}\left(\overline{R_K}[x]\right)$ under $f(x)=x^2$ is a continuum and properly contains  $\overline{R_{L}}[u]$ for all $u\in f^{-1}\left(\overline{R_K}[x]\right)$.

In section 2 we prove  Theorems \ref{closure} and \ref{connected-fiber}. In section 3 we prove Theorem \ref{invariant-fiber} and then use it to prove Theorem \ref{invariance}. In section 4 we give more examples that are related to our study.

\section{On the closed relation $\overline{R_K}$ and the connectedness of its fibers}
%To Prove Theorems \ref{closure} and \ref{connected-fiber}}

In this section we characterize $\overline{R_K}$ and show that each of its fibers $\overline{R_K}[x]$ is a continuum.  This establishes Theorems \ref{closure} and \ref{connected-fiber}. We also find a connection between the fibers of $\overline{R_K}$ and the prime end impressions of $D$, where $D$ is a simply connected domain with $K=\widehat{\bbC}\setminus D$.

In our characterization of $\overline{R_K}$, we shall frequently use the notation $U(J_1,J_2)$ for any disjoint Jordan curves $J_1$ and $J_2$, which denotes  the component of $\widehat{\bbC}\setminus(J_1\cup J_2)$ whose boundary is $J_1\cup J_2$.

We also need a  closed relation $\asymp_K$ defined in the following way.
Two points $x,y\in K$ are related under $\asymp_K$ provided that either $x=y$ or for  any $\displaystyle r\in\left(0,\frac12\rho(x,y)\right)$ the difference $K\setminus(D_r(x)\cup D_r(y))$ has infinitely many components intersecting both $\partial D_r(x)$ and $\partial D_r(y)$. Here $\rho$ denotes the spherical distance on $\widehat{\bbC}$.

Clearly, for $0<r<\frac12\rho(x,y)$ we can find $x_r\in\partial D_r(x)$ and $y_r\in\partial D_r(y)$ with $(x_r,y_r)\in R_K$. This implies that $\asymp_K\subset\overline{R_K}$, when these two relations are considered as subsets of $K\times K$. To obtain Theorem \ref{closure}, we only need to establish the reverse containment $\overline{R_K}\subset\  \asymp_K$.

The following lemma implies that $\asymp_K$ is closed.
\begin{lemma}\label{asymp_K}
For two points $x\ne y\in K$ to be related under $\asymp_K$ it is necessary and sufficient that there exist two sequences of Jordan domains $\{D_n': n\ge1\}$ and $\{D_n'': n\ge1\}$, with $\{x\}=\bigcap D_n'$ and $\{y\}=\bigcap D_n''$, such that  for all $n\ge1$ the next two requirements are both satisfied:
\begin{itemize}
\item that $\overline{D_{n+1}'}\subset D_n'$ and $\overline{D_{n+1}''}\subset D_n''$.
\item that $K\setminus(D_n'\cup D_n'')$ has infinitely many components intersecting both $\partial D_n'$ and $\partial D_n''$.
\end{itemize}
\end{lemma}
The necessity part is trivial. The sufficiency part is also direct. Actually, for  $0<r<\frac12\rho(x,y)$ we can find two Jordan domains $D_n'$ and $D_n''$ whose closures  are respectively contained in $D_r(x)$ and $D_r(y)$. Clearly, if a component of $K\setminus(D_n'\cup D_n'')$ intersects both $\partial D_n'$ and $\partial D_n''$, then it  contains a component of $K\setminus(D_r(x)\cup D_r(y))$ that intersects $\partial D_r(x)$ and $\partial D_r(y)$ both.

Since $\asymp_K$ is a closed relation contained in $\overline{R_K}$, Theorem \ref{closure} follows from the following.

\begin{theorem}\label{closure-b}
Two points $x\ne y\in K$ with $(x,y)\in R_K$ are related under $\asymp_K$.
\end{theorem}
\begin{proof}
Let $J_1\ni x$ and $J_2\ni y$ be two disjoint Jordan curves such that $K\cap\overline{U(J_1,J_2)}$ has an infinite sequence of distinct components $P_n$ intersecting both $J_1$ and $J_2$, whose limit under the Hausdorff distance $P_\infty=\lim\limits_{n\rightarrow\infty}P_n$  contains $\{x,y\}$.
Let $P_0$ be the component of $K\cap\overline{U(J_1,J_2)}$ with $P_0\supset P_\infty$.
Assume that $P_0\cap P_n=\emptyset$ for all $n\ge1$.

Fix an open arc $\beta\subset U(J_1,J_2)$, with $\overline{\beta}\cap K=\emptyset$, such that the endpoints belong to $J_1$ and $J_2$ respectively. See for instance \cite[Lemma 3.6 and Remark 3.7]{LLY-2019}.  Slightly thickening $\beta$, we obtain a Jordan domain $W\subset U(J_1,J_2)$ such that   $\overline{W}\cap K=\emptyset$ and $\partial W$ intersects $U(J_1,J_2)$ at  two open arcs $\beta_0,\beta_1$ with $\overline{\beta_0}\cap\overline{\beta_1}=\emptyset$. Therefore, $\Delta_\infty=U(J_1,J_2)\setminus\overline{W}$ is a Jordan domain whose boundary consists of  $\beta_0, \beta_1$ and two arcs $\gamma_1\subset J_1$ and $\gamma_2\subset J_2$. Clearly, $K\cap\overline{U(J_1,J_2)}=K\cap\overline{\Delta_\infty}$. Moreover,  $\overline{\Delta_\infty}\setminus P_\infty$ has two special components, say $W_0$ and $W_1$, satisfying $\beta_i\subset W_i$. With no loss of generality, we may assume that  $P_n\subset W_1$ for infinitely many $n$.
%Assume with no loss of generality that $P_n\subset B_1$ for all $n\ge1$.

Given  $k\ge2$, we pick two Jordan domains $D_k'\ni x$ and $D_k''\ni y$ of diameter $<\frac1k|x-y|$ that are respectively separated by $J_1$ and $J_2$ into exactly two Jordan domains. Then there is an infinite set $\mathcal{N}$ of integers such that every $P_n$ with $n\in\mathcal{N}$ is a subset of $W_1$ and intersects the two Jordan curves $J_k'=\partial D_k'$ and $J_k''=\partial D_k''$ both. For each $n\in\mathcal{N}$, we pick a component $B_n$ of $P_n\setminus(D_k'\cup D_k'')$ that intersects $J_k'$ and $J_k''$ both.
%By going to an appropriate subset, we may assume that  $\displaystyle \lim\limits_{n\in\mathcal{N},n\rightarrow\infty}B_n=B_\infty$ under the Hausdorff distance.
%Since $\max\{|x-x_k|,|y-y_k|\}<\frac1k$ and $k\ge2$ is arbitrary,
The only issue is to prove the following.
\begin{lemma}\label{coincidence_lemma}
Every $B_n$ with $n\in\mathcal{N}$ is a component of $K\cap\overline{U\left(J_k',J_k''\right)}$. Therefore, there are  $x_k\in J_k'$ and  $y_k\in J_k''$ with $(x_k,y_k)\in R_K$.
\end{lemma}
Indeed, for any $n_0\in\mathcal{N}$ the set $\Delta_\infty\setminus\left(\overline{D_k'}\cup\overline{D_k''}\cup B_{n_0}\right)$ has two special components $D_0$ and $D_1$, with $\beta_i\subset\partial D_i$, such that $D_0\supset B_n$ for infinitely many $n\in\mathcal{N}$. This is the case since $P_\infty\setminus(D_k'\cup D_k'')$ has a component that is contained in $D_0$ and intersects the two Jordan curves $J_k'$ and $J_k''$ both. By \cite[Lemma 3.6(3)]{LLY-2019},  for every $B_n\subset D_0$ we may pick two arcs $\alpha_n,\alpha_n'\subset D_0$ each of which connects a point on $J_k'$ to a point on $J_k''$. See Figure \ref{good-locations}, in which the endpoints of $\alpha_n$ marked by $a_n$ and $b_n$ while those of $\alpha_n'$ by $a_n'$ and $b_n'$.
\begin{figure}[ht]
\begin{center}\vspace{-0.382cm}
\begin{tikzpicture}[scale=1.618,x=1cm,y=1cm]
\draw  (0,0) rectangle (6.25,4);
\coordinate [label=90:$x$] (x) at (3,4); \coordinate [label=270:$y$] (y) at (3,0);
\draw[line width=2pt,color=blue] (3.5,1.54) --(3.5,2.46);
\draw[line width=2pt,color=blue] (4.25,1.03) --(4.25,2.97);
\draw (x)[fill=black]circle(0.03); \draw  (y)[fill=black]circle(0.03);
\draw (0,2)node[right]{$\beta_0$};   \draw  (6.25,2)node[left]{$\beta_1$};
\node at (3.75,2.0) {$B_n$};
\node at (4.6,2.0) {$B_{n_0}$};
\node at (3,0.6) {\small$D_k''\cap U(J_1,J_2)$}; \node at (3,3.4) {\small$D_k'\cap U(J_1,J_2)$};
\draw[line width=1.5pt,color=gray] (1.382,4) arc(180:360:1.618cm); \draw[line width=1.5pt,color=gray] (1.382,0) arc(180:0:1.618cm);
\draw[line width=0.5pt,color=blue] (4.0,1.26)--(4.0,2.74);
\fill[blue] (4.0,1.275) circle(0.03); \fill[blue] (4.0,2.725) circle(0.03);
\node at (3.9,1.2) {$b_n$}; \node at (3.9,2.9) {$a_n$};
\draw[line width=0.5pt,color=blue] (3.25,1.6)--(3.25,2.4);
\fill[blue] (3.25,1.6) circle(0.03); \fill[blue] (3.25,2.4) circle(0.03);
\node at (3.25,1.33) {$b_n'$}; \node at (3.25,2.67) {$a_n'$};
\end{tikzpicture}
\end{center}
\vskip -0.618cm
\caption{Relative locations of $B_{n_0}, B_n$ and the two arcs $\alpha_n, \alpha_n'$.}\label{good-locations}
\end{figure}
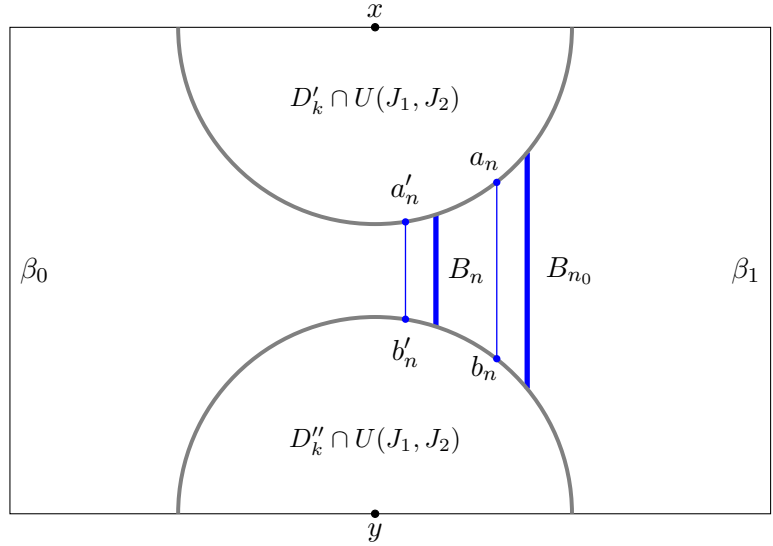

Let $D_n\subset D$ be the unique Jordan domain whose boundary consists of $\alpha_n$,  $\alpha_n'$, and two other arcs, that are respectively contained in $\partial D_k'$ and $\partial D_k''$. Then $B_n$ is a component of $K\cap\overline{D_n}$ hence is also a component of $K\cap\overline{U(J_k',J_k'')}$. Here we have $K\cap\overline{U(J_k'.J_k'')}=K\setminus\left(D_k'\cup D_k''\right)$.
\end{proof}

In order to obtain Theorem \ref{connected-fiber}, we just need to prove the following.
\begin{theorem}\label{connected-fiber-b}
If $(x,y)\in\overline{R_K}$ then $\overline{R_K}[x]$ has a subcontinuum that contains $\{x,y\}$.
\end{theorem}
\begin{proof} Assume $x\ne y$. We fix a strictly decreasing sequence $\{r_n\}$ with $\displaystyle 3n\cdot r_n\in\left(0,\rho(x,y)\right)$ and choose $P_n(n\ge1)$ as follows. First, by Theorem \ref{closure} and the assumption $(x,y)\in\overline{R_K}$ we know that $\displaystyle K\setminus\left(D_{r_1}(x)\cup D_{r_1}(y)\right)$ has infinitely many components intersecting $\partial D_{r_1}(x)$ and $\partial D_{r_1}(y)$ both. Pick such a component $P_1$. If $P_1,\ldots,P_n$ are chosen, we consider the  components of $\displaystyle K\setminus\left(D_{r_{n+1}}(x)\cup D_{r_{n+1}}(y)\right)$ intersecting both $\partial D_{r_{n+1}}(x)$ and $\partial D_{r_{n+1}}(y)$. There are infinitely many of them and we can pick one, say $P_{n+1}$,  that is disjoint from $P_1\cup\cdots\cup P_n$. By going to an appropriate subsequence, we may assume that $P_\infty=\lim\limits_{n\rightarrow\infty}P_n$ under the Hausdorff distance.

Since $y\in P_\infty$,  we just need to verify that $(x,z)\in \overline{R_K}$ holds for all $z\in P_\infty\setminus\{x,y\}$.

Given $z\in P_\infty\setminus\{x,y\}$, we fix $N_z>1$ with
$\frac{2}{N_z}\le\min\{\rho(x,z),\rho(y,z),\rho(x,y)\}$.
Then  $\overline{D_{r_n}(x)}$, $\overline{D_{r_n}(y)}$, and $\overline{D_{r_n}(z)}$ are pairwise disjoint for all $n\ge N_z$.
For  $n\ge N_z$, let $A_n$ be the annulus  with $\partial A_n=\partial D_{r_n}(x)\cup \partial D_{r_n}(y)$.
We will find $a_n\in\partial D_{r_n}(x)$ and $b_n\in\partial D_{r_n}(z)$ with $(a_n,b_n)\in R_K$.

By the same method as used in Theorem \ref{closure-b}, we can pick open arcs $\beta_0,\beta_1\subset A_n$ separating $A_n$ into two Jordan domains such that the closure of  one of them, say $D_n$, contains $P_\infty$. See Figure \ref{conn-fiber}.
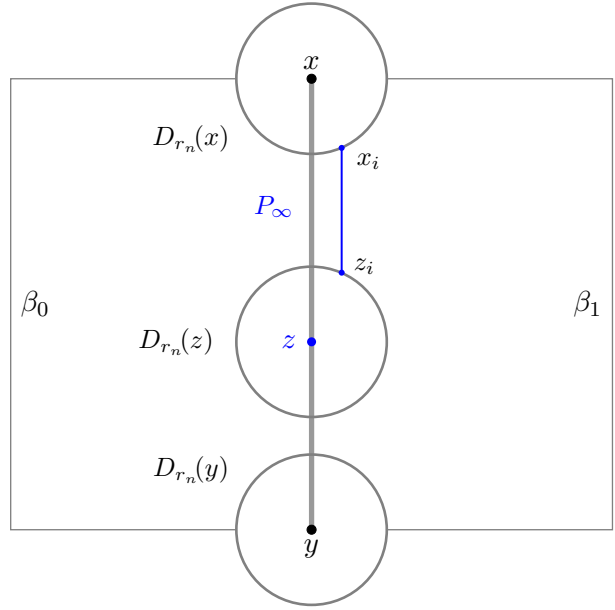
\begin{figure}[ht]
\centering
\vskip -0.5cm
\begin{tikzpicture}[scale=2.0,x=1cm,y=1cm]
%\draw  (0,0) rectangle (6.25,3);
\draw[line width=0.5pt,color=gray] (2.5,0) --(1,0)--(1,3)--(2.5,3);
\draw[line width=0.5pt,color=gray] (3.5,0) --(5,0)--(5,3)--(3.5,3);
\draw[line width=2pt,color=gray!80] (3,0) --(3,3);
\coordinate [label=90:$x$] (x) at (3,3); \coordinate [label=270:$y$] (y) at (3,0);
\draw (x)[fill=black]circle(0.03); \draw  (y)[fill=black]circle(0.03);
\draw (1,1.5)node[right]{$\beta_0$};   \draw  (5.0,1.5)node[left]{$\beta_1$};
\node at (2.2,0.4) {\small$D_{r_n}\!(y)$};
\node at (2.1,1.25) {\small$D_{r_n}\!(z)$};
\node at (2.2,2.6) {\small$D_{r_n}\!(x)$};
\draw[line width=1.0pt,color=gray] (3.5,3) arc(0:360:0.5cm);
\draw[line width=1.0pt,color=gray] (3.5,0) arc(0:360:0.5cm);
\node[blue] at (2.75,2.15) {\small$P_\infty$};
\node[blue] at (2.85,1.25) {$z$};
\fill[blue] (3.0,1.25) circle(0.03);
\draw[line width=1.0pt,color=gray] (3.5,1.25) arc(0:360:0.5cm);
\draw[line width=0.75pt,color=blue] (3.2,2.55) --(3.2,1.7);
\fill[blue] (3.2,2.54) circle(0.02);
\fill[blue] (3.2,1.71) circle(0.02);
\node at (3.38,2.45) {\small$x_i$};  \node at (3.35,1.75) {\small$z_i$};
\end{tikzpicture}
\vskip -0.5cm
\caption{Relative locations of $x_i,z_i,P_\infty$ and $D_{r_n}(x),D_{r_n}(y),D_{r_n}(z)$.}\label{conn-fiber}
\vskip -0.382cm
\end{figure}
Here $K\cap\left(\overline{\beta_0}\cup\overline{\beta_1}\right)=\emptyset$ and $U_n=D_n\setminus\overline{D_{r_n}(z)}$ is an open annulus.
Pick $k_n\ge1$ such that $P_i\subset\overline{D_n}$ for all $i\ge k_n$. For each $i\ge k_n$, we can pick a component $Q_i$ of $P_i\setminus D_{r_n}(z)$ that intersects $\partial D_{r_n}(x)$. This is doable, since $P_i\cap\partial D_{r_n}(x)\ne\emptyset$ for all $i\ge k_n$. Then we shall have $Q_i\cap\partial D_{r_n}(z)\ne\emptyset$ by the Cut Wire Theorem \cite[p.72, Theorem 5.2]{Nadler92}. Since $\{Q_i: i\ge k_n\}$ has a subsequence converging to a limit continuum $Q_\infty$ under the Hausdorff distance,  $(a_n,b_n)\in R_K$ holds for any $a_n\in\left(Q_\infty\cap \partial D_{r_n}(x)\right)$ and any $b_n\in\left(Q_\infty\cap \partial D_{r_n}(z)\right)$. This implies  that $(x,z)\in\overline{R_K}$, since we have $\lim\limits_{n\rightarrow\infty}\rho(a_n,x)=0=\lim\limits_{n\rightarrow\infty}\rho(b_n,z)$.
\end{proof}

The rest part of this section is about simply connected domains $D\subset\widehat{\bbC}$ that have more than one boundary point. In Theorem \ref{fiber-and-imp} we will obtain a basic connection between the prime end impressions of $D$ and certain fibers of $\overline{R_K}$ with $K=\widehat{\bbC}\setminus D$.  This connection is of its own interest and will be used in Example \ref{inverse_of_fibers}. On the other hand, by Theorem \ref{fiber-and-imp} and \cite[Lemma 17]{BCO11} one can easily infer a special case of \cite[Theorem 5]{LLY-2019}, when $K\subset\widehat{\bbC}$ is a full continuum.

The prime end theory is introduced by Carath\'eodory \cite{Caratheodory13-a} and his work is in German. For the sake of convenience, we recall some notions and basic results from  \cite{UrsellYoung51}.

An open or half-open {\bf path} in $D$ is a continuous map $\alpha: I\rightarrow D$ or its image, with  $I=(0,1)$ or  $I=[0,1)$. Such a path is called a {\bf simple path} or an {\bf arc}, if $\alpha$ is injective. The {\bf initial limit set} of $\alpha$ consists of all points $z\in\overline{D}$ such that $z=\lim\limits_{n\rightarrow\infty}\alpha(t_n)$ for some decreasing sequence $t_n\in(0,1)$ with $t_n\rightarrow 0$; and the {\bf final limit set} is  the collection of  $z\in\overline{D}$ such that $z=\lim\limits_{n\rightarrow\infty}\alpha(t_n)$ for an increasing sequence $t_n\in(0,1)$ with $t_n\rightarrow 1$. If the initial (respectively, final) limit set consists of a single point $z_0$, we call $z_0$ the {\bf initial point} (respectively, the {\bf endpoint}) of $\alpha$.

A {\bf crosscut} in  $D$ means an open arc $\alpha$ whose initial point and endpoint are different and both belong to the boundary  $\partial D$.
%A  {\bf returncut} in $U$ is an open arc with two requirements: (1) the initial point belongs to $\partial D$ and coincides with the endpoint, (2)  its closure does not separate $\partial D$.
A {\bf chain of crosscuts} consists of an infinite sequence of crosscuts $\{q_n\}$ with disjoint closures such that $F_n\supset F_{n+1}$ for all $n\ge1$, where $F_n$ is the component of $D\setminus q_n$  containing $q_{n+1}$ and will be called the {\bf shadow} of $q_n$.

A {\bf C-transformation} of $D$ is a homeomorphism of $D$ onto $\mathbb{D}=\{z\in\mathbb{C}: |z|<1\}$ such that the image  of any crosscut of $D$ is a crosscut of $\mathbb{D}$ and that the endpoints of those crosscuts of $\mathbb{D}$ are dense on $\partial\mathbb{D}$. It is known \cite[(4.1) to (4.3)]{UrsellYoung51} that a C-transformation is the composition of a conformal transformation of $D$ onto $\mathbb{D}$ with a homeomorphism of $\overline{\mathbb{D}}$ onto itself.
%Given a C-transformation $T: U\rightarrow\mathbb{D}$, the following holds:  \begin{itemize} \item[(a)] If a path $p$ in $U$ possesses an end-point, so does $T(p)$. \item If $r$ is a return-cut of $D$ then $T(r)$ is a return-cut of $\mathbb{D}$. \item Given any arc $s\subset\partial\mathbb{D}$, there exist an arbitrary finite number of simple paths $p\subset D$ possessing end-points which are all distinct, such that the transformed paths $T(p)$ have end-points in $s$. \end{itemize}

We fix a C-transformation $T$, a conformal transformation of $D$ onto $\mathbb{D}$.
Two half-open paths $\alpha,\beta: I\rightarrow D$ are said to {\bf possess the same prime end}
if the endpoints of $T(\alpha),T(\beta)$ coincide and belong to $\partial\mathbb{D}$. Therefore, every prime end of $D$ corresponds to a point $p\in\partial\mathbb{D}$, under the fixed C-transformation $T$. By \cite[(4.5) to (4.7)]{UrsellYoung51},  two half-open paths $\alpha,\beta$ possess the same prime end if and only if they cross a chain of crosscuts $\{q_n\}$ whose closures are concentric circular arcs tending to the common center.

By \cite[(4.8)]{UrsellYoung51}, among the half-open paths possessing the same prime end there is one whose limit set is maximal (i.e. includes the others) and one whose limit set is minimal (i.e. is contained in all the others). The maximal limit set is called the {\bf impression} of the prime end while the minimal limit set is called the {\bf principal set}, whose elements are often referred to as {\bf principle points}  of the prime end. If the prime end corresponds to $p\in\partial\mathbb{D}$, we denote by $\text{Imp}(p)$ the impression and by $\prod(p)$ the principal set.

In the sequel, we consider the continuum $K=\widehat{\bbC}\setminus D$ and relate the prime end impressions of $D$ to  certain fibers of $\overline{R_K}$.

\begin{theorem}\label{fiber-and-imp}
Given $p\in\partial\mathbb{D}$ and $x\in\prod(p)$, we have $\text{Imp}(p)\subset\overline{R_K}[x]$.
\end{theorem}
\begin{proof}
By  Theorem \ref{closure}, it suffices to show that for any $y\in \text{Imp}(p)$ and $0<r<\frac12\rho(x,y)$ the set $K\setminus\left(D_r(x)\cup D_r(y)\right)$ has infinitely many components intersecting $\partial D_r(x)$ and $\partial D_r(y)$ both.

Note that $K\setminus\left(D_r(x)\cup D_r(y)\right)=K\cap A_r$, where  $A_r=\widehat{\bbC}\setminus\left(D_r(x)\cup D_r(y)\right)$ is a closed annulus. Therefore, by \cite[Lemma 3.8]{LLY-2019}, we only need to show that $A_r\setminus K$ has infinitely many components intersecting $\partial D_r(x)$ and $\partial D_r(y)$ both.

Since $x\in \prod(p)$, by \cite[(4.9)]{UrsellYoung51} we can fix a chain of crosscuts $\{q_n\}$, whose closures are concentric circular arcs with center $x$ and radius $r_n\rightarrow0$, such that $\{q_n\}$ is crossed by every path possessing the prime end corresponding to $p$.
By \cite[(4.8)]{UrsellYoung51}, we may further fix a half-open path $\alpha=\alpha(t)(0\leq t<1)$ which crosses the chain $\{q_n\}$ and has limit set $\text{Imp}(p)$.
Here we may assume with no loss of generality that  $q_n\subset D_r(x)$ for all $n\ge1$.
Notice that $\text{Imp}(p)=\bigcap \overline{F_n}$, where  $F_n$ denotes the shadow of $q_n$. Let $E_n$ be the other component of $D\setminus q_n$.

Now we fix $t_1\in(0,1)$ with $x_1=\alpha(t_1)\in D_r(x)$ and consider $\alpha_1=\alpha(t)(t_1\leq t<1)$, which is a half-open path intersecting $D_r(y)$. Let $b_1>t_1$ be the minimal number with $y_1'=\alpha(b_1)\in\partial D_r(y)$ and   $a_1\in(t_1,b_1)$ the maximal number with $x_1'=\alpha(a_1)\in\partial D_r(x)$. Let $n_1\ge1$ be the minimal integer such that $\beta_1=\alpha([a_1,b_1])\subset E_{n_1}$.
See Figure \ref{fiber-impression-1}.
\begin{figure}[ht]
\vspace{-0.382cm}
\begin{center}
\begin{tikzpicture}[scale=1.382, x=0.6cm,y=0.6cm]
\draw[fill,gray!38.2]  (-7.5,1.5) circle (1.75);
\draw[gray] (-7.5,1.5) circle (1.76);
\draw[fill,gray!38.2]  (3.5,1.5) circle (1.75);
\draw[gray](3.5,1.5)circle (1.76);

\draw[very thick, purple]  plot[smooth, tension=.7] coordinates {(-7.6,2.6) (-2,4.2) (3.35,3.25)};
\draw[very thick, purple]  plot[smooth, tension=.7] coordinates {(-7.0,2.0) (-2,3.0) (2.0,2.4)};
%
%\draw[very thick, purple]  plot[smooth, tension=.7] coordinates {(-5.9,0.85) (-2,0.95) (1.9,0.85)};
% \beta_1
\draw[fill]  (-7.6,2.6) circle (0.1); \draw[fill]  (-6.6,3.0) circle (0.1);
\draw[fill]  (3.35,3.25) circle (0.1);
\node[] at (-8.0,2.6) {$x_1$}; \node[] at (-6.65,3.4) {$x_1'$}; \node[] at (3.5,3.65) {$y_1'$};
% \beta_1
\draw[fill]  (-7.0,2.0) circle (0.1); \draw[fill]  (-5.93,2.28) circle (0.1);
\draw[fill]  (2.0,2.4) circle (0.1);
\node[] at (-7.5,2.0) {$x_2$}; \node[] at (-5.5,2.0) {$x_2'$}; \node[] at (2.2,2.1) {$y_2'$};
\node[very thick, purple] at (-2.0,3.8) {$\beta_1$};
\node[very thick, purple] at (-2.0,2.6) {$\beta_2$};
%\node[very thick, purple] at (-2.0,0.5) {$\beta_\infty$};

\node at (-10,1.5) {$D_r(x)$};
\node at (6,1.5) {$D_r(y)$};

\node[very thick, purple] at (-2,1.9) {\LARGE$\vdots$};
\node[very thick, purple] at (-2,1.3) {\LARGE$\vdots$};

\end{tikzpicture}
\vskip -0.382cm
\caption{Relative locations of $\beta_j\subset F_{n_j}\subset E_{n_{j+1}}$  and the points $x_j, x_j',y_j'$.}\label{fiber-impression-0}
\end{center}
\vskip -0.25cm
\end{figure}
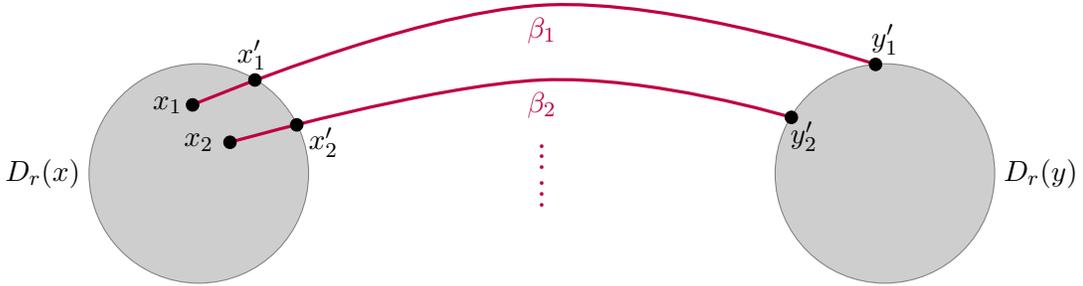

Let $n_0=0$ and $F_{n_0}=D$.

Suppose that $n_j$ and $\alpha_j\supset\beta_j\supset\{x_j',y_j'\}$, with $x_j'=\alpha(a_j)$ and $y_j'=\alpha(b_j)$,  have been chosen for $j\in\{1,\ldots,k\}$ such that $\beta_{j}$ is contained in $F_{n_{j-1}}\cap E_{n_{j}}$ for $1\le j\le k$. We may continue to find $\alpha_{k+1}\supset \beta_{k+1}\supset \{x_{k+1}',y_{k+1}'\}$ and $n_{k+1}$ in the following way.  First,  pick $t_{k+1}\in(b_k,1)$ with $x_{k+1}=\alpha(t_{k+1})\in D_r(x)$ such that  $\alpha_{k+1}=\alpha(t)(t_{k+1}\leq t<1)$ is a half-open path contained in $F_{n_k}$. This is doable, since $\beta_k=\alpha([a_k,b_k])\subset E_{n_k}$. Second, as $\alpha_{k+1}$ intersects $D_r(y)$, we further choose $b_{k+1}\in(t_{k+1},1)$ to be the minimal number with $y_{k+1}'=\alpha(b_{k+1})\in\partial D_r(y)$ and $a_{k+1}\in(t_{k+1},b_{k+1})$ the maximal one with $x_{k+1}'=\alpha(a_{k+1})\in\partial D_r(x)$. Let $n_{k+1}$ be the minimal integer such that $\beta_{k+1}=\alpha([a_{k+1},b_{k+1}])$ is contained in $E_{n_{k+1}}$. Clearly, $n_{k+1}>n_k$.

Inductively, we can find a strictly increasing  sequence of integers $\{n_j: j\ge1\}$ and an infinite sequence of disjoint paths $\{\beta_j\}$, with $\beta_j=\alpha([a_j,b_j])$, such that  $\beta_{j+1}$ is contained in $F_{n_j}\cap E_{n_{j+k}}$ for all $j,k\ge1$. Recall that $F_{n_j}\subset F_{n_i}$ for all $i<j$ and $q_{n_j}\subset D_r(x)$ always holds. It follows that no component of $A_r\setminus K$ contains $\beta_i$ and $\beta_j$ both, since every component of $A_r\setminus K$ is a subset of $D$ and since every arc $\gamma\subset D$ connecting $\beta_i$ to $\beta_j$ necessarily intersects $q_{n_j}$. Therefore, $A_r\setminus K$ has  infinitely many components that intersect  $\partial D_r(x)$ and $\partial D_r(y)$ both.
\end{proof}

\begin{example}\label{fiber-and-imp-0}
Let $D\subset\Delta=\{u+v\textbf{i}: 0<u, v<1\}$ be given as in \cite[p.4, Fig.6]{UrsellYoung51}, considered as a subset of $\widehat{\bbC}$. Let $K=\widehat{\bbC}\setminus D$. Then $\partial D$ is the union of $\partial\Delta$ with the segments
\[A_n=\left\{u+v\textbf{i}: u=\frac{1}{2^{2n-1}}, 0\le v\le \frac34\right\} \quad
\text{and}\quad
B_n=\left\{u+v\textbf{i}: u=\frac{1}{2^{2n}}, \frac14\le v\le 1\right\}.
\]
See Figure \ref{fiber-impression-1}.
 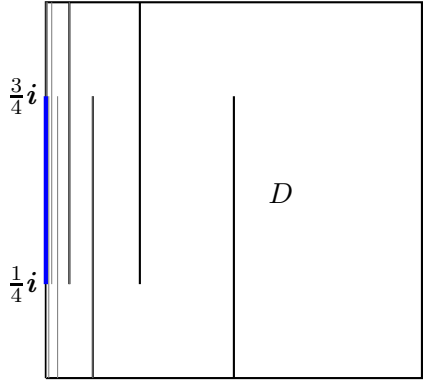
\begin{figure}[ht] \vskip -0.5cm \begin{center}
\begin{tikzpicture}[scale=1.25,x=1cm,y=1cm]
%\fill[gray!38.2,thick]  (-1,-1) rectangle (5,5);
%\fill[white]  (0.01,0.01) rectangle (3.99,3.99);
%
\draw[thick](0,0)--(4,0)--(4,4)--(0,4)--(0,0);
\draw[black,thick] (2,0) --(2,3); \draw[black,thick] (1,4) --(1,1);
\draw[black,thick] (1/2,0) --(1/2,3); \draw[black,thick] (1/4,4) --(1/4,1);
\foreach \i in {2,8,32,128}
{\draw[gray,thin] (1/\i,0) --(1/\i,3); \draw[gray,thin] (0.5/\i,4) --(0.5/\i,1);
}
\draw[blue,ultra thick] (0,1) --(0,3);
\draw  (2.5,1.75) node[above]{$D$};
\draw  (0,1) node[left]{$\frac14\textbf{i}$};
\draw  (0,3) node[left]{$\frac34\textbf{i}$};
\end{tikzpicture} \end{center} \vskip -0.5cm
\caption{A simple depiction of $D$ and the segments $A_n, B_n$.}\label{fiber-impression-1} \vskip -0.25cm \end{figure}
Notice that there is a prime end whose impression equals $\{t\textbf{i}: 0\le t\le1\}$. If $x=t\textbf{i}$ and $t\in[0,\frac14)$ then $\overline{R_K}[x]=\{s\textbf{i}: 0\le s\le \frac34\}$. Similarly, if $x=t\textbf{i}$ and $t\in(\frac34,1]$ then $\overline{R_K}[x]=\{s\textbf{i}: \frac14\le s\le 1\}$. In both cases, $x$ is not a principal point and $\text{Imp}(p)\subset\overline{R_K}[x]$ does not hold. However, every $x=t\textbf{i}$ with $t\in[\frac14,\frac34]$ is a principal point and $\text{Imp}(p)=\overline{R_K}[x]$.
\end{example}

\section{Branched coverings $f: \widehat{\bbC}\rightarrow\widehat{\bbC}$ send atoms of $f^{-1}(K)$ onto atoms of $K$}
This section gives the proofs for Theorem \ref{invariant-fiber} and Theorem \ref{invariance}.
In the sequel, $K\subset\widehat{\bbC}$ is assumed to be a compactum and  $f: \widehat{\bbC}\rightarrow\widehat{\bbC}$ a branched covering. Moreover, we set $L=f^{-1}(K)$ and let
$f^{-1}\left(\Dc_K^{PC}\right)$ denote the collection of all those continua $N$ that is a component of $f^{-1}(d)$ for some $d\in\Dc_K^{PC}$. By \cite[p.44, Theorem 3.21]{Nadler92},  $\left\{f^{-1}(d): d\in\Dc_K^{PC}\right\}$ is an usc decomposition of $L$. By \cite[p.278, Lemma 13.2]{Nadler92}, $f^{-1}\left(\Dc_K^{PC}\right)$ is an usc decomposition into subcontinua.
We respectively denote by $\sim_K$ and $\sim_L$ the Sch\"onflies equivalences on $K$ and $L$.

We first obtain the following.

\begin{theorem}\label{fiber-into}
 The containment $f\left(\overline{R_L}[x]\right)\subset\overline{R_{K}}[f(x)]$ holds for all $x\in L$. Consequently,  every atom of $L$ is sent by $f$ {\bf into} an atom of $K$.
\end{theorem}

\begin{proof}%[{\bf Proof for Theorem \ref{fiber-into}}]
Fix $x\ne y\in L$ with $(x,y)\in\overline{R_{L}}$.
If on the contrary $a=f(x)$ and $b=f(y)$ were not related under $\overline{R_K}$, there would exist some $r>0$ such that the closures of $D_r(a)$ and $D_r(b)$ are disjoint and that $K\setminus\left(D_r(a)\cup D_r(b)\right)$ has at most finitely many components intersecting $\overline{D_r(a)}$ and $\overline{D_r(b)}$ both, say $N_1,\ldots, N_l$.

Choose  $\epsilon>0$ with $\overline{D_\epsilon(x)}\subset f^{-1}\left(D_r(a)\right)$ and $\overline{D_\epsilon(y)}\subset f^{-1}\left(D_r(b)\right)$. Then $\overline{D_\epsilon(x)}\cap \overline{D_\epsilon(y)}=\emptyset$ and $L\setminus\left(D_\epsilon(x)\cup D_\epsilon(y)\right)$ has infinitely many components that intersect $\overline{D_\epsilon(x)}$ and $\overline{D_\epsilon(y)}$ both.
Denote these components as $P_1,P_2,\ldots$.
By the Cut Wire Theorem \cite[p.72, Theorem 5.2]{Nadler92}, the difference $P_i\setminus\left[f^{-1}(D_r(a))\cup f^{-1}(D_r(b))\right]$ for each $i\ge1$ has a component $Q_i$ that intersects $f^{-1}\left(\overline{D_r(a)}\right)$ and $f^{-1}\left(\overline{D_r(b)}\right)$ both. Note that $Q_i$ is also a component of
\[X_r=L\setminus\left[f^{-1}(D_r(a))\cup f^{-1}(D_r(b))\right]=f^{-1}\left[K\setminus(D_r(a)\cup D_r(b))\right].
\]
Clearly, for all $i\ge1$ the image $f(Q_i)$ is a continuum and is contained in one of the continua $N_1,\ldots, N_l$. Thus every $Q_i$ is a component of the inverse $f^{-1}(N)$ for some component $N$ of $K\setminus(D_r(a)\cup D_r(b))$. Note that $N=f(Q_i)$, since $f$ is a branched covering. Also note that $N$ intersects $\overline{D_r(a)}$ and $\overline{D_r(b)}$ both hence coincides with one of $N_1,\ldots, N_l$. As $f$ is a branched covering, there exists $1\le j\le l$ with $f(Q_i)=N_j$. Therefore, there is some $N_j$ such that $f(Q_i)=N_j$ for infinitely many $i$. This is absurd, since $f^{-1}(z)$ is a finite set for all $z\in\widehat{\bbC}\setminus B_f$, where the branch set $B_f$ is finite.
\end{proof}

Then we deal with Theorem \ref{invariant-fiber}. By Theorem \ref{fiber-into}, it suffices to obtain the following.

\begin{theorem}\label{fiber-onto}
The containment $f\left(\overline{R_L}[x]\right)\supset\overline{R_K}[f(x)]$ holds for $x\in L$.
\end{theorem}

Recall that an {\bf elementary region $W(J_1,\ldots,J_k)$} \cite[p.38]{Whyburn64} means a domain in the plane that is bounded by $k\ge2$ disjoint Jordan curves $J_1,\ldots,J_k$, such that every $J_i$ is a component of $\partial W$. In Theorem \ref{good-R}, we will use elementary regions to induce a criterion for two points $x\ne y\in K$ to be related under $R_K$. This criterion is useful when we prove Theorem \ref{fiber-onto}.
%An annulus is just an elementary region bounded by two disjoint Jordan curves.

\begin{theorem}\label{good-R}
Two points $x\ne y\in K$ are related under $R_K$ if and only if there is an elementary region $W=W(J_1,\ldots,J_k)$ with the following  {\bf$(\star)$-condition}: there are two integers, say $i_1=1$ and $i_2=2$, such that $\overline{W}\cap K$ has infinitely many components $P_n$ intersecting both $J_1$ and $J_2$ and the limit $P_\infty=\lim\limits_{n\rightarrow\infty}P_n$ under the Hausdorff distance satisfies $x\in(J_1\cap P_\infty)$ and $y\in(J_2\cap P_\infty)$.
\end{theorem}
\begin{proof}
Assume that $k\ge3$. Let $U=U(J_1,J_2)$ be the annulus bounded by $J_1$ and $J_2$, which necessarily contains each of $J_3,\ldots,J_k$.
Let $V_i(1\le i\le k)$ be the component of $\widehat{\bbC}\setminus\overline{W}$ with $\partial V_i=J_i$. Assume that $\infty\in V_1$ and call a domain $V\subset\widehat{\bbC}$ {\bf bounded} if it does not contain $\infty$.

If no $J_i$ with $i\ge3$ can be connected to $J_1\cup J_2$ by a simple arc $\alpha\subset (W\setminus K)$, then one can infer that every $V_i$ with $i\ge3$ is contained in a bounded component of $\widehat{\bbC}\setminus P$ for some component $P$ of $\overline{W}\cap K$. Let $K^*=\bigcup_{i=3}^k\overline{V_i}$. Then, all but $k-2$ of the continua $P_n$ are disjoint from $K^*$ thus are components of $\overline{W}\cap (K\cup K^*)$. Each of those $P_n$, with $P_n\cap K^*=\emptyset$, is a subset of $\overline{U(J_1,J_2)}\cap K$ and hence a component of $\overline{U(J_1,J_2)}\cap K$.
This indicates that $(x,y)\in R_K$.

In the sequel we suppose that one of the curves $J_i$ with $i\ge3$, say $J_k$, can be connected to $J_1\cup J_2$ by an arc $\alpha\subset (W\setminus K)$. With no loss of generality, we may assume that $\alpha$ has one endpoint on $J_2$ and the other on $J_k$. Notice that $\alpha$ may be thickened to a Jordan domain $D\subset(W\setminus K)$ whose closure does not intersects $K$. Moreover, the boundary $\partial D$ consists of two arcs $\alpha', \alpha''\subset (W\setminus K)$, together with one arc on $J_k$ and another on $J_2$.
Now we can construct a Jordan curve $J_2'\subset(J_2\cup J_k\cup\alpha'\cup\alpha'')$ and a new elementary region $W_1$ bounded by $k-1$ disjoint Jordan curves $J_1, J_2', J_3,\ldots,J_{k-1}$. See Figure \ref{tunnel} for the region bounded by $J_2'$.
\begin{figure}[ht]
\vspace{-0.2cm}
\begin{center}
\begin{tikzpicture}[scale=0.55]
\footnotesize
 \pgfmathsetmacro{\xone}{0}
 \pgfmathsetmacro{\xtwo}{21}
 \pgfmathsetmacro{\yone}{0}
 \pgfmathsetmacro{\ytwo}{10}

% the interior of \Gamma_2'
\foreach \p in {0.05,0.10,...,5.95}
    \draw[green,very thin] (\xtwo/7+\p,\ytwo/7) -- (\xtwo/7+\p,6*\ytwo/7);

\foreach \p in {0.05,0.10,...,5.95}
    \draw[green,very thin] (4*\xtwo/7+\p,\ytwo/7) -- (4*\xtwo/7+\p,6*\ytwo/7);

\foreach \p in {0.05,0.10,...,2.95}
    \draw[green,very thin] (3*\xtwo/7+\p,3.8*\ytwo/7) -- (3*\xtwo/7+\p,4.2*\ytwo/7);

\draw[gray,very thick] (\xone,\yone) -- (\xtwo,\yone) --  (\xtwo,\ytwo) --
 (\xone,\ytwo) --  (\xone,\yone);

\draw[gray,very thick] (\xone+\xtwo/7,\yone+\ytwo/7) -- (\xone+3*\xtwo/7,\yone+\ytwo/7) --  (\xone+3*\xtwo/7,\yone+6*\ytwo/7) --
(\xone+\xtwo/7,\yone+6*\ytwo/7) -- (\xone+\xtwo/7,\yone+\ytwo/7);

\draw[gray,very thick] (\xone+4*\xtwo/7,\yone+\ytwo/7) -- (\xone+6*\xtwo/7,\yone+\ytwo/7) --  (\xone+6*\xtwo/7,\yone+6*\ytwo/7) --
(\xone+4*\xtwo/7,\yone+6*\ytwo/7) -- (\xone+4*\xtwo/7,\yone+\ytwo/7);

\draw[blue,very thick] (\xone+4*\xtwo/7,\yone+4*\ytwo/7) -- (\xone+3*\xtwo/7,\yone+4*\ytwo/7);

\draw[red,very thick] (\xone+4*\xtwo/7,\yone+4.2*\ytwo/7) -- (\xone+3*\xtwo/7,\yone+4.2*\ytwo/7);

\draw[red,very thick] (\xone+4*\xtwo/7,\yone+3.8*\ytwo/7) -- (\xone+3*\xtwo/7,\yone+3.8*\ytwo/7);

% the curves \gamma_i
 \draw[blue] (\xone,\ytwo) node[anchor=north west] {$J_1$};
 \draw[blue] (1.2*\xtwo/7,\ytwo/7) node[anchor=north east] {$J_2$};
 \draw[blue] (4.2*\xtwo/7,\ytwo/7) node[anchor=north east] {$J_k$};
 \draw[red] (3.5*\xtwo/7,3.85*\ytwo/7) node[anchor=north] {$\alpha'$};
 \draw[red] (3.5*\xtwo/7,4.2*\ytwo/7) node[anchor=south] {$\alpha''$};

\end{tikzpicture}
\end{center}
\vskip -0.5cm
\caption{Relative location of $J_1, J_2, J_k, \alpha'$ and $\alpha''$ along with $W^*$.}\label{tunnel}
\vskip -0.25cm
\end{figure}
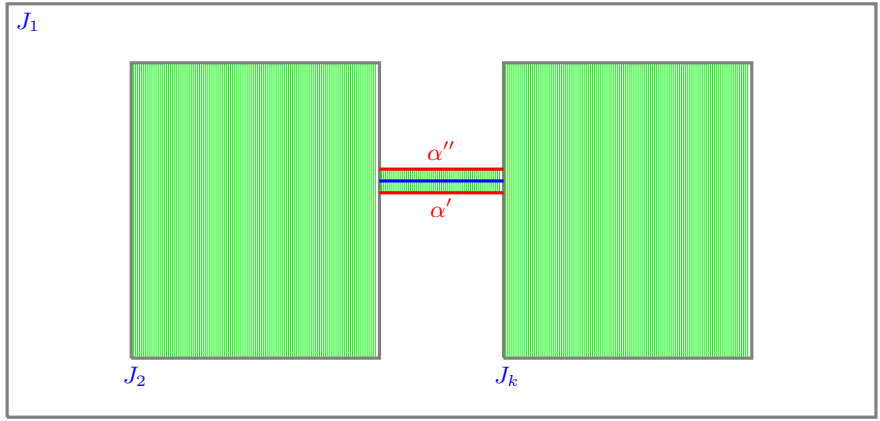

Note that all but finitely many of the above continua $P_n$ are also components of $\overline{W_1}\cap K$.
Also note that $W_1$ is again an elementary region, bounded by $k-1$ disjoint Jordan curves,  that satisfies the $(\star)$-condition. Repeating this for at most $k-2$ times, we will obtain an annulus bounded by two disjoint Jordan curves that satisfies the $(\star)$-condition. We are done. \end{proof}

% \begin{lemma}\label{circle-inverse} Given a finite-to-one open map  $f: \widehat{\bbC}\rightarrow \widehat{\bbC}$ and $z\in\widehat{\bbC}$, there exist two constants $\epsilon, r>0$ such that each component of $f^{-1}\left(\overline{D_r(z)}\right)$  lies in one of the disks $D_{\epsilon}(u)$ with $u\in f^{-1}(z)$ and that each component of $f^{-1}\left(\partial D_r(z)\right)$ is a Jordan curve. \end{lemma}

\begin{proof}[{\bf Proof of Theorem \ref{fiber-onto}}]
Set $a=f(x)$. Given $b\in\overline{R_K}[a]$, we will find $y\in\overline{R_L}[x]$ with $f(y)=b$.

Since $f^{-1}(\{a,b\})$ is a finite set, we may fix $\epsilon>0$ such that the closures of
$D_\epsilon(u)$ and $D_\epsilon(v)$ with  $u\ne v\in f^{-1}(\{a,b\})$ are disjoint. Since $f$ is a branched covering,  we may further fix $r>0$ and a sequence $\{r_n: n\ge2\}$ with $0<r_n<\frac{r}{n}$ such that for any $z\in\{a,b\}$ the next two properties are satisfied:
(1) each component of $f^{-1}\left(\overline{D_{r_n}(z)}\right)$  lies in one of the disks $D_{\epsilon/n}(u)$ with $f(u)=z$; (2)  each component of $f^{-1}\left(\partial D_{r_n}(z)\right)$ is a Jordan curve.
%By \cite[p.280, Theorem 13.5]{Nadler92} and the openness of $f$, we further choose a number $r_n\in(0,\frac{r}{n})$ for each $n\ge2$ such that $f^{-1}\left(D_{r_n}(a)\right)$ lies in $\displaystyle \bigcup_{f(u)=a}D_{\epsilon/n}(u)$ and $f^{-1}\left(D_{r_n}(b)\right)$ in $\displaystyle \bigcup_{f(v)=b}D_{\epsilon/n}(v)$.
% \begin{equation}\label{control-disk} f^{-1}\left(D_{r_n}(a)\right)  \subset\left(\bigcup_{f(u)=a}D_{\epsilon/n}(u)\right) \quad \text{and}\quad f^{-1}\left(D_{r_n}(b)\right) \subset\left(\bigcup_{f(v)=b}D_{\epsilon/n}(v)\right). \end{equation}
Let $U_n(n\ge2)$ be the component of $f^{-1}\left(\overline{D_{r_n}(a)}\right)$ that contains $x$. Then $\displaystyle W_n=\widehat{\bbC}\setminus\left(f^{-1}\left(\overline{D_{r_n}(a)}\right)\cup f^{-1}\left(\overline{D_{r_n}(b)}\right)\right)$ is an elementary region bounded by finitely many disjoint Jordan curves, one of which is $\partial U_n$.

Given $n\ge1$, we know that $K\setminus\left(D_{r_n}(a)\cup D_{r_n}(b)\right)$ has infinitely many components, say $\{P_{n,k}: k\ge1\}$, that intersect  $\partial D_{r_n}(a)$ and $\partial D_{r_n}(b)$ both.
Every $f^{-1}\left(P_{n,k}\right)$ has at least one component, say $M_{n,k}$, such that  $M_{n,k}\cap U_n\ne\emptyset$. Clearly, $M_{n,k}$ intersects one component of $f^{-1}\left(\overline{D_{r_n}(b)}\right)$, say $B_{n,k}$, that contains a point $y_{n,k}\in f^{-1}(b)$.
As $f^{-1}(b)$ is a finite set,  there exists $k_n\ge1$ and some $y_n\in f^{-1}(b)$ such that $y_{n,k}=y_n$ and $B_{n,k}=B_{n,k_n}$ for infinitely many $k\ge1$. Here $y_n\in B_{n,k_n}$. Going to an appropriate subsequence, we may assume that $y_{n,k}\equiv y_n(k\ge1)$ and that $\lim\limits_{k\rightarrow\infty} M_{n,k}=M_{n,\infty}$ under the Hausdorff distance.

By applying Theorem \ref{good-R} to the elementary region $W_n$,  we can find   $u_n\in\left(\overline{U_n}\cap M_{n,\infty}\right)$ and $v_n\in\left(B_{n,k_n}\cap M_{n,\infty}\right)$ with $(u_n,v_n)\in R_L$. Here we have $y_n,v_n\in B_{n,k_n}$. Going to an appropriate subsequence, we may assume that $y_n\equiv y(n\ge1)$ for some $y\in f^{-1}(b)$. Then $\lim\limits_{n\rightarrow\infty}u_n=x$ and $\lim\limits_{n\rightarrow\infty}v_n=y$. Moreover, we have $(x,y)\in\overline{R_L}$ and $f(y)=\lim\limits_{n\rightarrow\infty}f(v_n)=b$.
\end{proof}

%The 3 citations below may be removed later on. \textcolor{blue}{\footnotesize\input{Nadler92_280_13-5}} \textcolor{blue}{\footnotesize\input{Nadler92_284_13-14}} \textcolor{blue}{\footnotesize\input{Nadler92_287_13-22}}

Now we have all the ingredients to prove Theorem \ref{invariance} as follows.
\begin{proof}[{\bf Proof for Theorem \ref{invariance}}]
Using the closed equivalence $\sim_L$, we define an equivalence  $\approx$ on $K$, such that  $x,y\in K$ are related under $\approx$ if and only if $\pi_L\left(f^{-1}(x)\right)=\pi_L\left(f^{-1}(y)\right)$.
This happens if and only if the union of all the classes  $[u]_{\sim_L}$ of $\sim_L$ with $f(u)=x$ equals the union of all the classes  $[v]_{\sim_L}$ with $f(v)=y$.

There are two observations. First, by Theorem \ref{fiber-onto} we can easily check that $\approx$ contains $\overline{R_K}$. Second, by the upper semi-continuity of $\sim_L$ and the openness of $f$, we can further show that $\approx$ is a closed equivalence hence contains $\sim_K$.
Actually, if $(x,y)\notin\approx$  then $\pi_L\left(f^{-1}(x)\right)\ne \pi_L\left(f^{-1}(y)\right)$. Without losing generality, we may assume that  $\pi_L(u_0)\notin\pi_L\left(f^{-1}(y)\right)$ for some $u_0\in f^{-1}(x)$. This implies that $[u_0]_{\sim_L}$ is disjoint from all the classes of $\sim_L$ that intersects $f^{-1}(y)$. Since $\sim_L$ is an usc decomposition, we can find disjoint open sets $U\ni u_0$ and $V\supset f^{-1}(y)$ both of which are saturated with respect to $\sim_L$, so that every class of $\sim_L$ intersecting $U$ (or $V$) is entirely contained in $U$ (or in $V$). The openness of $f$ then ensures that $f(U)\ni x=f(u_0)$ and $f(V)\ni y$ are disjoint open sets such that $U\times V$ is disjoint from $\approx$. Thus $K\!\times\!K\setminus\approx$ \ is open.

Finally, by Theorem \ref{fiber-into}, the containment $\left[f(u)\right]_{\sim_K}\supset f\left([u]_{\sim_L}\right)$ holds for all $u\in L$. On the other hand, if $y\in [f(u)]_{\sim_K}$ then  $f(u)\approx y$ and we can find $w\in f^{-1}(y)$ with $[w]_{\sim_L}=[u]_{\sim_L}$. It follows that $y\in f\left([u]_{\sim_L}\right)$. By the flexibility of $y$ we immediately have $[f(u)]_{\sim_K}\subset f\left([u]_{\sim_L}\right)$. This establishes Theorem \ref{invariance}.
\end{proof}

\begin{remark}
Let $\phi(x)=\pi_K\circ f(x)$ for all $x\in L$, where $\pi_K:K\rightarrow\Dc_K^{PC}$ is the Peano projection.
Let $\phi_1:L\rightarrow f^{-1}\left(\Dc_K^{PC}\right)$ be the natural projection and  $\phi_2(N)=d$ for all $N\in f^{-1}\left(\Dc_K^{PC}\right)$, where $d\in\Dc_K^{PC}$ is the unique element such that $N$ is a component of $f^{-1}(d)$. Then $\phi_1$ is monotone and $\phi_2$ is finite-to-one.
By \cite[p.141, Factor Theorem (4.1)]{Whyburn42}, $\phi=\phi_2\circ \phi_1$ is the unique {\bf monotone-light factorization} of $\phi: L\rightarrow \Dc_K^{PC}$.
By Theorem \ref{invariance},  $\phi_1$ is topologically equivalent to $\pi_L$ and $\phi_2$ is a finite-to-one map from the Peano model of $L$ onto that of $K$.
\end{remark}

\section{Examples}

This section gives some planar compacta $K$ for which the core decompositions $\mathcal{D}_K^{PC}$ are known. For those $K$, we demonstrate how the fibers  $\overline{R_K}[x]$  are connected to the fibers of $\sim_K$, which are elements of $\mathcal{D}_K^{PC}$ and are called atoms of $K$.

Hereafter, let  $\mathcal{C}$ denote the Cantor's ternary set.
Example \ref{invariance-destroyed} shows that $L$ in Theorem \ref{invariance} can not be replaced by any of its proper subsets.

\begin{example}\label{invariance-destroyed}
Let $A=\{re^{{\bf i}\theta}:1\leq r\leq 2,\pi\leq\theta\leq 2\pi\}$ and $L=\mathcal{K}_0\cup A$, where $\mathcal{K}_0=\left\{(1+r)e^{{\bf i}\theta}:\ r\in \mathcal{C},0\leq \theta\leq \pi\right\}$.
See Figure \ref{Exp1}.
%for a finite approximation of $\mathcal{K}_0\setminus A^o$, which consists of all the non-degenerate atoms of $L$.
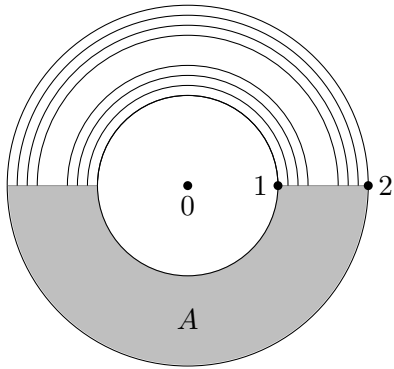
\begin{figure}[ht]
\begin{tikzpicture}[scale=1.2]
\fill[gray!50] (0,0)  circle (12.0ex);
\draw[gray!80,very thin] (-2,0) -- (2,0);
\fill[white] (0,0)  circle (6.0ex);
\fill[white] (-2.1,0.01)--(2.1,0.01)--(2.1,2)--(-2.1,2)--(-2.1,0.01);

\fill[black] (0,0) node[below]{$0$} circle (0.3ex);
\fill[black] (1,0) node[left]{$1$} circle (0.3ex);
\fill[black] (2,0) node[right]{$2$} circle (0.3ex);
\draw (0,-1.25) node[below] {$A$};

%\fill[gray] (4,0) node[right]{$4$} circle (0.3ex);

\draw (2,0) arc (0:-180:2 cm);
\draw (1,0) arc (0:-180:1 cm);

%\draw (4,0) arc (0:180:4 cm);
\draw (1,0) arc (0:180:1 cm);

\draw (1,0) arc (0:180:1 cm);
\draw (2,0) arc (0:180:2 cm);
%\draw (3,0) arc (0:180:3 cm);
\draw (1.333,0) arc (0:180:1.333 cm);
\draw (1.667,0) arc (0:180:1.667 cm);

%\draw (3.333,0) arc (0:180:3.333 cm); \draw (3.667,0) arc (0:180:3.667 cm);

\draw (1.111,0) arc (0:180:1.111 cm);
\draw (1.222,0) arc (0:180:1.222 cm);
\draw (1.778,0) arc (0:180:1.778 cm);
\draw (1.889,0) arc (0:180:1.889 cm);

%\draw (3.111,0) arc (0:180:3.111 cm);  \draw (3.222,0) arc (0:180:3.222 cm); \draw (3.778,0) arc (0:180:3.778 cm); \draw (3.889,0) arc (0:180:3.889 cm);

\end{tikzpicture}
\vskip -0.382cm
\caption{A simple depiction of $\mathcal{K}_0$, $A$ and $L$.}\label{Exp1}
\end{figure}
Let $f(z)=z^2$. Then $K=f(L)$ is just the closed annulus $\{z: 2\le|z|\le4\}$. Moreover, we shall have $L \subsetneqq f^{-1}(K)$ and  $\overline{R_L}=\sim_L$. Note that every atom of $L$ is either a semi-circle contained in $\mathcal{K}_0$ or a singleton in $A\setminus\mathcal{K}_0$, while every atom of $K$ is a singleton. Therefore, $L$ has uncountably non-degenerate atoms $d$, each of which is a semi-circle contained in $\mathcal{K}_0$, such that $f(d)$ consists of uncountably many atoms of $K$.
\end{example}

%In Example \ref{inverse_of_fibers}, we give two continua $K,L$ such that (1) $L=f^{-1}(K)$ with $f(z)=z^2$ and (2) for some $z_0$ the preimage $N=f^{-1}\left(\overline{R_K}[z_0]\right)$ is a continuum that properly contains $\overline{R_L}[u]$ for all $u\in N$.

Example \ref{inverse_of_fibers} is related to Theorem \ref{invariant-fiber}.

\begin{example}\label{inverse_of_fibers}
Let $\alpha_k=\left\{t_ke^{{\bf i}\theta}:0\leq \theta\leq\pi\right\}$ and $\beta_k=\left\{s_ke^{{\bf i}\theta}:\pi\leq \theta\leq2\pi\right\}$  for all  $k\geq0$, where $t_k=4^{\frac{1}{2^k}}$ and $s_k=\frac12\left(t_k+t_{k+1}\right)$. Then $K=\overline{\bigcup_{k\ge0}(\alpha_k\cup\beta_k)}\cup[1,4]$ is a continuum containing  the unit circle $\partial\mathbb{D}$. Let $L=f^{-1}(K)$ with $f(z)=z^2$. See  Figure \ref{Exp2}, in which the end points of $\alpha_k$ (respectively, of $\beta_k$) are marked by $\pm z_k$ (respectively, $\pm w_k$). Then $\overline{R_K}[1]=\partial\mathbb{D}=f^{-1}\left(\overline{R_K}[1]\right)$. However, $\overline{R_L}[z]\subsetneqq\partial\mathbb{D}$ for each $z\in \partial\mathbb{D}$.
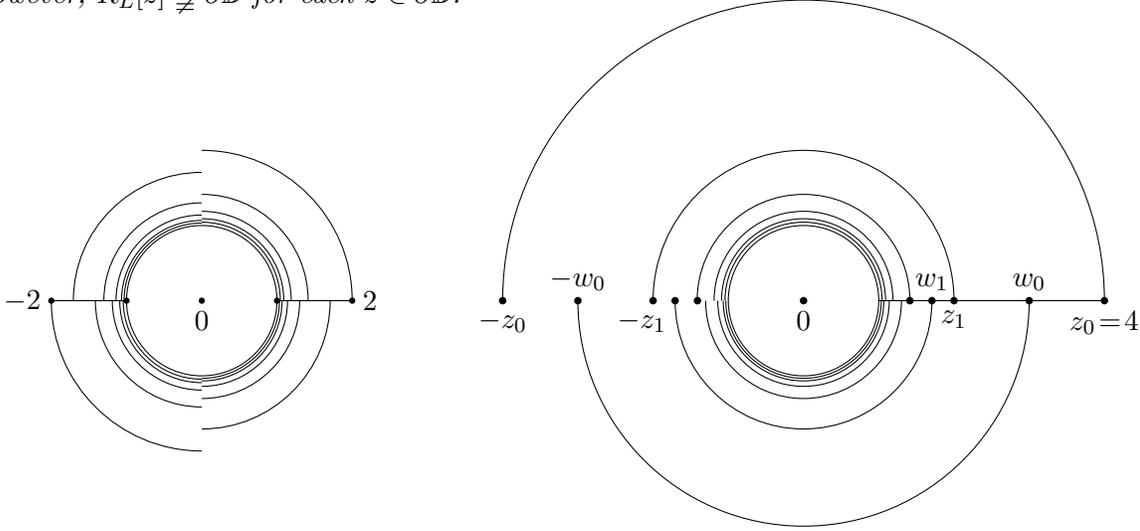
\begin{figure}[ht]
\vspace{-0.618cm}
\begin{center}
\begin{tikzpicture}[x=1cm,y=1cm,scale=1.0]
\draw (1,0)->(2,0);
\draw (-1,0)->(-2,0);
\draw (0,0)circle(1);

\foreach \i in {1,...,5}
 \pgfmathsetmacro{\j}{2^\i}
  \pgfmathsetmacro{\s}{1/\j}
  \pgfmathsetmacro{\r}{4^\s}
\draw (\r,0) arc (0:90:\r cm);

\foreach \i in {1,...,5}
 \pgfmathsetmacro{\j}{2^\i}
  \pgfmathsetmacro{\s}{1/\j}
  \pgfmathsetmacro{\r}{4^\s}
\draw (-\r,0) arc (180:270:\r cm);

\foreach \i in {1,...,5}
  \pgfmathsetmacro{\j}{2^\i}
   \pgfmathsetmacro{\s}{1/\j}
   \pgfmathsetmacro{\r}{4^\s}
   \pgfmathsetmacro{\p}{4^(\s/2)}
  \pgfmathsetmacro{\t}{(\r+\p)/2}
\draw (0,\t) arc (90:180:\t cm);

\foreach \i in {1,...,5}
  \pgfmathsetmacro{\j}{2^\i}
   \pgfmathsetmacro{\s}{1/\j}
   \pgfmathsetmacro{\r}{4^\s}
   \pgfmathsetmacro{\p}{4^(\s/2)}
  \pgfmathsetmacro{\t}{(\r+\p)/2}
\draw (0,-\t) arc (270:360:\t cm);

\fill[black] (0,0) node[below]{$0$} circle (0.25ex);
\fill[black] (1,0)  circle (0.25ex); %node[left]{$1$}
\fill[black] (2,0) node[right]{$2$} circle (0.25ex);
\fill[black] (-1,0)  circle (0.25ex); %node[right]{$-1$}
\fill[black] (-2,0) node[left]{$-2$} circle (0.25ex);

%========================================================

\draw (9,0)->(12,0);
\draw (8,0)circle(1);

\foreach \i in {0,...,5}
 \pgfmathsetmacro{\j}{2^\i}
  \pgfmathsetmacro{\s}{1/\j}
  \pgfmathsetmacro{\r}{4^\s}
\draw (8+\r,0) arc (0:180:\r cm);

\foreach \i in {0,...,5}
  \pgfmathsetmacro{\j}{2^\i}
   \pgfmathsetmacro{\s}{1/\j}
   \pgfmathsetmacro{\r}{4^\s}
   \pgfmathsetmacro{\p}{4^(\s/2)}
  \pgfmathsetmacro{\t}{(\r+\p)/2}
\draw (8+\t,0) arc (0:-180:\t cm);

\fill[black] (4,0) node[below]{$-z_0$} circle (0.3ex);
\fill[black] (12,0) node[below]{$z_0\!=\!4$} circle (0.3ex);
\fill[black] (10,0) node[below]{$z_1$} circle (0.3ex);
\fill[black] (6,0)  circle (0.3ex);
\fill[black] (5.85,0.0) node[below]{$-z_1$};
\fill[black] (9.414,0) circle (0.3ex); %node[below]{$z_2$}
\fill[black] (6.59,0) circle (0.3ex);
\fill[black] (11,0) node[above]{$w_0$} circle (0.3ex);
\fill[black] (5,0) node[above]{$-w_0$} circle (0.3ex);
\fill[black] (9.707,0) node[above]{$w_1$} circle (0.3ex);
\fill[black] (6.293,0) circle (0.3ex);
\fill[black] (8,0) node[below]{$0$} circle (0.3ex);
\end{tikzpicture}

\end{center}
\vskip -0.618cm
\caption{A finite approximation of $L$ (left) and $K$ (right).}\label{Exp2}
\end{figure}
\end{example}

In the next five examples we give  special compacta $K$ and compare the equivalence $\sim_K$ with the iterates $\left(\overline{R_K}\right)^n(n\ge1)$. Here two points $x,y\in K$ are related under the $n$-th iterate $\left(\overline{R_K}\right)^n$ if and only if there exist $n+1$ points $z_1=x,z_2,\ldots,z_n,z_{n+1}=y$ such that $(z_i,z_{i+1})\in\overline{R_K}$ for $1\le i\le n$.
Each of those compacta $K$ falls into one of the following categories.
\begin{enumerate}
\item  $K$ itself is an atom of $K$ and $\overline{R_K}=\sim_K$. See Example \ref{indecomposable}.
\item  Every component of $K$ is an atom of $K$ and $\overline{R_K}=\sim_K$. See Example \ref{McMullen}.
\item  $K$ itself is an atom of $K$ and $\overline{R_K}\ne\overline{R_K}^2=\sim_K$. See Example \ref{teepee}.
%\item $K$ is an atom but it is not clear whether $\overline{R_K}[x]=K$ for some $x\in K$. See Example \ref{Cremer}.
\item $K$ is NOT an atom and $\overline{R_K}^{n+1}\ne \overline{R_K}^{n+2}=\sim_K$ for some $n\ge1$. See Example \ref{rotated_combs}.
\item $K$ has exactly one non-degenerate atom and $\overline{R_K}^n\ne\sim_K$ for all $n\ge1$. See Example \ref{witch_broom}.
\end{enumerate}

\begin{example}\label{indecomposable}
If $K\subset\mathbb{C}$ is an indecomposable continuum then  $\overline{R_K}[x]=K$ for all $x\in K$. Consequently, we have $\mathcal{D}_K^{PC}=\{K\}$ hence there is only one atom, the whole continuum $K$ itself.
\end{example}

%\begin{rem*} Example \ref{indecomposable} is related to Theorem \ref{fiber-and-imp} and \cite[Theorem 1.1]{CMR06}. There are two further questions of some interest. One is to find a decomposable continuum $K\subset\mathbb{C}$ such that $\overline{R_K}[x]=K$ for all $x\in K$. Such a continuum $K$ may be constructed by using two pseudo-arcs. The other is to find reasonable conditions under which a planar continuum $K$ is indecomposable if and only if $\overline{R_K}[x]=K$ for all $x\in K$. \end{rem*}

%Example \ref{McMullen} is about the Julia sets of McMullen maps $f(z)=z^2+\frac{\lambda}{z^3}$  with small $\lambda\in\bbC$.

\begin{example}\label{McMullen}
By \cite[Proposition 7.2]{McMullen88} we know that for all complex number $\lambda\ne0$ sufficiently small the Julia set $K$ of $f(z)=z^2+\frac{\lambda}{z^3}$ is homeomorphic with $\left\{re^{2\pi{\bf i}\theta}:\ r\in \mathcal{C},0\leq \theta\leq 2\pi\right\}$.
%where $\mathcal{C}$ is the Cantor's ternary set.
Such a compactum is called a {\bf Cantor set of circles}. Each component of $K$ is an atom.
Further discussions on  rational maps whose Julia sets are a Cantor set of circles can be found in \cite{QYY15}.
\end{example}

\begin{example}\label{teepee}
If $K$ is Cantor's teepee \cite[p.145]{SS-1995} then  $\overline{R_K}[p]=K$ for exactly one point $p\in K$. Moreover, the fiber $\overline{R_K}[x]$ is a simple arc for every  $x\in K\setminus\{p\}$. See  Figure \ref{Teepee_pic}.
\begin{figure}[ht]
\begin{center}\vskip -0.382cm
\begin{tikzpicture}[x=1cm,y=1cm,scale=1.1]
%% scale=1
%\draw[green,thin] (0,0) -- (7.29,0);
\foreach \i in {0,...,3}
{
\draw[gray,thick] (3.645,2) -- (0.09*\i,0);
\draw[gray,thick] (3.645,2) -- (0.54+0.09*\i,0);
\draw[gray,thick] (3.645,2) -- (1.62+0.09*\i,0);
\draw[gray,thick] (3.645,2) -- (2.16+0.09*\i,0);

\draw[gray,thick] (3.645,2) -- (4.86+0.09*\i,0);
\draw[gray,thick] (3.645,2) -- (5.4+0.09*\i,0);
\draw[gray,thick] (3.645,2) -- (6.48+0.09*\i,0);
\draw[gray,thick] (3.645,2) -- (7.02+0.09*\i,0);
}
\fill[black] (3.645,2) node[above]{$p$} circle (0.3ex);
\end{tikzpicture}
\end{center}\vskip -0.5cm
\caption{A simple representation of Cantor's Teepee.}\label{Teepee_pic}\vskip -0.5cm
\end{figure}
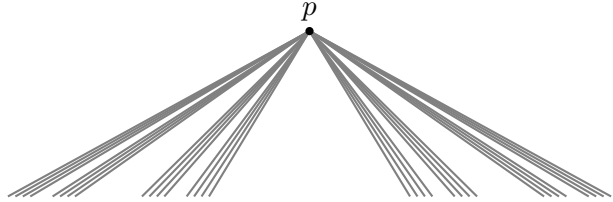
\end{example}

%\begin{example}\label{Cremer} Let $K$ be the Julia set of a quadratic polynomial $f$ with a Cremer fixed point $\alpha_0$. Let $\alpha_{-1}\ne \alpha_0$ be the other point lying in  $f^{-1}(\alpha_0)$. It is known \cite[Proposition II.11]{PM94} that $\alpha_0$ and $\alpha_{-1}$ are contained in a prime end impression of $U_\infty$, the unbounded Fatou component of $f$.  Let $d$ be the atom of $K$  containing $\alpha_0$. By Corollary \ref{invariant-core}, every component of $f^{-1}(d)$ is also an atom of $K$. Since every component of $f^{-1}(d)$ contains either $\alpha_0$ or $\alpha_{-1}$ and since $\{\alpha_0,\alpha_{-1}\}\subset d\subset f^{-1}(d)$, we see that $d=f^{-1}(d)$. As the Julia set $K$ is a minimal compact set that coincides with its inverse under $f$, we have $d=K$. In other words, $K$ itself is an atom. \end{example}

%\begin{theorem}\label{indecomposable_criterion} A full continuum $K\subset\widehat{\bbC}$ is indecomposable if and only if $\overline{R_K}[x]=K$ for all $x\in K$. \end{theorem}

\begin{example}\label{rotated_combs}
Let $K_1$ be the union of  $[0,1]\subset\bbC$ with $A=\left\{t+s\mathbf{i}: t\in\mathcal{C}, s\in[0,1]\right\}$ and $B=\left\{s+t\mathbf{i}: t\in\mathcal{C}, s\in[0,1]\right\}$.
%where $\mathcal{C}$ denotes the Cantor's ternary set.
Moreover, for all $n\ge1$ let $K_{2n+1}=\bigcup_{m=0}^n\left(K_1+m+m\textbf{i}\right)$ and $K_{2n}=K_{2n-1}\cup\left(A+n+n\textbf{i}\right)$. Then $\overline{R_{K_n}}^{n+1}\subsetneqq\overline{R_{K_n}}^{n+2}=\sim_{K_n}$ holds for all $n\ge1$. See  Figure \ref{combs}.
\begin{figure}[ht]
\begin{center}
\begin{tabular}{ccc}
\begin{tikzpicture}[x=1cm,y=1cm,scale=0.5]
\draw[thick] (0,0) -- (3,0);
\foreach \i in {0,...,3}
{
    \draw[gray,thick] (3*\i/27,0) -- (3*\i/27,3);
    \draw[gray,thick] (6/9+3*\i/27,0) -- (6/9+3*\i/27,3);
     \draw[gray,thick] (2+3*\i/27,0) -- (2+3*\i/27,3);
    \draw[gray,thick] (2+6/9+3*\i/27,0) -- (2+6/9+3*\i/27,3);
}
\foreach \i in {0,...,3}
{
    \draw[gray,thick] (3,3*\i/27) -- (6,3*\i/27);
    \draw[gray,thick] (3,6/9+3*\i/27) -- (6,6/9+3*\i/27);
     \draw[gray,thick] (3,2+3*\i/27) -- (6,2+3*\i/27);
    \draw[gray,thick] (3,2+6/9+3*\i/27) -- (6,2+6/9+3*\i/27);
}
\fill[black] (0,0) node[below]{$0$} circle (0.5ex);
\fill[black] (3,0) node[below]{$1$} circle (0.5ex);
%\fill[black] (6,3) node[above]{$2+{\bf i}$} circle (0.3ex);

\end{tikzpicture}
&
\begin{tikzpicture}[x=1cm,y=1cm,scale=0.5]
\draw[thick] (0,0) -- (3,0);
\foreach \i in {0,...,3}
{
    \draw[gray,thick] (3*\i/27,0) -- (3*\i/27,3);
    \draw[gray,thick] (6/9+3*\i/27,0) -- (6/9+3*\i/27,3);
     \draw[gray,thick] (2+3*\i/27,0) -- (2+3*\i/27,3);
    \draw[gray,thick] (2+6/9+3*\i/27,0) -- (2+6/9+3*\i/27,3);
}
\fill[black] (0,0) node[below]{$0$} circle (0.5ex);
\fill[black] (3,0) node[below]{$1$} circle (0.5ex);
%\fill[black] (6,3) node[right]{$2+{\bf i}$} circle (0.3ex);

%%
\foreach \i in {0,...,3}
{
    \draw[gray,thick] (3,3*\i/27) -- (6,3*\i/27);
    \draw[gray,thick] (3,6/9+3*\i/27) -- (6,6/9+3*\i/27);
     \draw[gray,thick] (3,2+3*\i/27) -- (6,2+3*\i/27);
    \draw[gray,thick] (3,2+6/9+3*\i/27) -- (6,2+6/9+3*\i/27);
}

%%%
\foreach \i in {0,...,3}
{
    \draw[gray,thick] (3+3*\i/27,3) -- (3+3*\i/27,6);
    \draw[gray,thick] (3+6/9+3*\i/27,3) -- (3+6/9+3*\i/27,6);
     \draw[gray,thick] (3+2+3*\i/27,3) -- (3+2+3*\i/27,6);
    \draw[gray,thick] (3+2+6/9+3*\i/27,3) -- (3+2+6/9+3*\i/27,6);
}

\end{tikzpicture}
&
\begin{tikzpicture}[x=1cm,y=1cm,scale=0.5]
\draw[thick] (0,0) -- (3,0);
\foreach \i in {0,...,3}
{
    \draw[gray,thick] (3*\i/27,0) -- (3*\i/27,3);
    \draw[gray,thick] (6/9+3*\i/27,0) -- (6/9+3*\i/27,3);
     \draw[gray,thick] (2+3*\i/27,0) -- (2+3*\i/27,3);
    \draw[gray,thick] (2+6/9+3*\i/27,0) -- (2+6/9+3*\i/27,3);
}
\fill[black] (0,0) node[below]{$0$} circle (0.5ex);
\fill[black] (3,0) node[below]{$1$} circle (0.5ex);
%\fill[black] (9,3) node[right]{$3+{\bf i}$} circle (0.3ex);

%%
\foreach \i in {0,...,3}
{
    \draw[gray,thick] (3,3*\i/27) -- (6,3*\i/27);
    \draw[gray,thick] (3,6/9+3*\i/27) -- (6,6/9+3*\i/27);
     \draw[gray,thick] (3,2+3*\i/27) -- (6,2+3*\i/27);
    \draw[gray,thick] (3,2+6/9+3*\i/27) -- (6,2+6/9+3*\i/27);
}

%%%
\foreach \i in {0,...,3}
{
    \draw[gray,thick] (3+3*\i/27,3) -- (3+3*\i/27,6);
    \draw[gray,thick] (3+6/9+3*\i/27,3) -- (3+6/9+3*\i/27,6);
     \draw[gray,thick] (3+2+3*\i/27,3) -- (3+2+3*\i/27,6);
    \draw[gray,thick] (3+2+6/9+3*\i/27,3) -- (3+2+6/9+3*\i/27,6);
}

%%%%
\foreach \i in {0,...,3}
{
    \draw[gray,thick] (6,3+3*\i/27) -- (9,3+3*\i/27);
    \draw[gray,thick] (6,3+6/9+3*\i/27) -- (9,3+6/9+3*\i/27);
     \draw[gray,thick] (6,3+2+3*\i/27) -- (9,3+2+3*\i/27);
    \draw[gray,thick] (6,3+2+6/9+3*\i/27) -- (9,3+2+6/9+3*\i/27);
}

\end{tikzpicture}
\end{tabular}
\end{center}
\vskip -0.385cm
\caption{A finite approximation of $K_2$ (left), $K_3$ (middle), and $K_4 $ (right).}\label{combs}
\vskip -0.5cm
\end{figure}
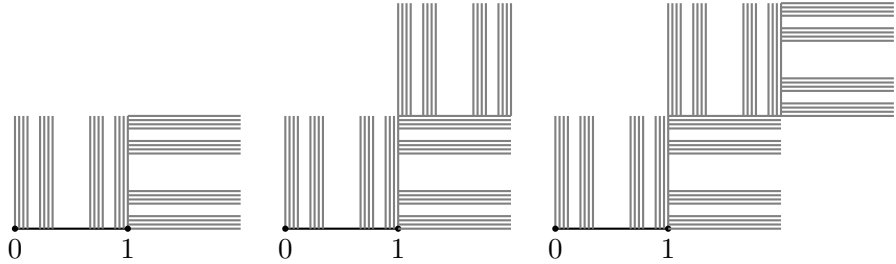
Similar continua are discussed in \cite[Example 2.12]{JLL16} and \cite[Example 7.6]{LL-2018} for different purposes.
\end{example}

\begin{example}\label{witch_broom}
Let $K\subset\mathbb{C}$ be the Witch's Broom \cite[p.84, Figure 5.22]{Nadler92}. See Figure \ref{broom}. Then the segment $[0,1]\subset\mathbb{R}\subset\mathbb{C}$ is the only non-degenerate atom of $K$. We have the following observations:
(1)  $\overline{R_K}[x]=\{x\}$ for all points $x\in K\setminus(0,1]$; (2)  $\overline{R_K}[x]$ is a nondegenerate subarc of $(0,1]$ for all points $x\in (0,1]$, and (3) $\left(\overline{R_K}\right)^n[x]\subsetneqq[0,1]$ for all $x\in [0,1]$ and $n\ge1$.
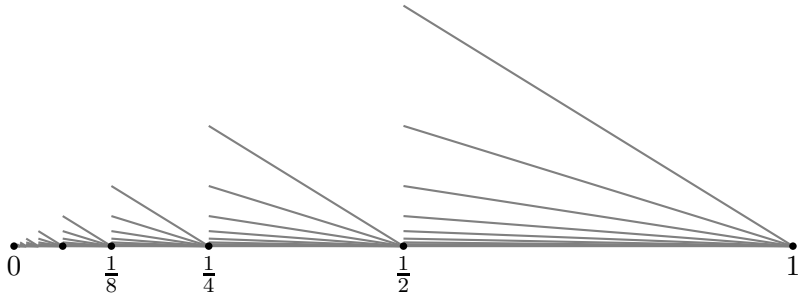
\begin{figure}[ht]
\vspace{-0.cm}
\begin{center}
\begin{tikzpicture}[x=1.618cm,y=1cm,scale=1]
%% scale=1

\foreach \i in {0,...,8}
{
\draw[gray,thick] (6.4,0) -- (3.2,3.2/2^\i);
\draw[gray,thick] (3.2,0) -- (1.6,1.6/2^\i);
\draw[gray,thick] (1.6,0) -- (0.8,0.8/2^\i);
\draw[gray,thick] (0.8,0) -- (0.4,0.4/2^\i);
\draw[gray,thick] (0.4,0) -- (0.2,0.2/2^\i);
\draw[gray,thick] (0.2,0) -- (0.1,0.1/2^\i);
\draw[gray,thick] (0.1,0) -- (0.05,0.05/2^\i);
\draw[gray,thick] (0.05,0) -- (0.025,0.025/2^\i);
}
\draw[gray,thick] (0,0) -- (6.4,0);
\fill[black] (6.4,0) node[below]{$1$} circle (0.3ex);
\fill[black] (3.2,0) node[below]{$\frac12$} circle (0.3ex);
\fill[black] (1.6,0) node[below]{$\frac14$} circle (0.3ex);
\fill[black] (0.8,0) node[below]{$\frac18$} circle (0.3ex);
\fill[black] (0.4,0)  circle (0.3ex);
%\fill[black] (0.2,0)  circle (0.3ex);
\fill[black] (0,0) node[below]{$0$} circle (0.3ex);
\end{tikzpicture}
\end{center}
\vskip -0.618cm
\caption{A finite approximation of the Witch's Broom.}\label{broom}
\vskip -0.25cm
\end{figure}
\end{example}

{\bf Acknowledgements}.
The first and second named authors are respectively supported by the Chinese National Natural Science Foundation Projects 11871483 and 11901011. The second author  is also supported by Science and Technology Projects of Guangzhou 202102020480 and Guangdong Basic and Applied Basic Research Foundation 2021A1515010242.
 The authors want to thank anonymous referees for suggestions that help to simplify some results and their proofs.
During private communications, the authors have received  suggestions from Alexander Blokh and Lex Oversteegen, concerning how to associate to any rational function $f$ with Julia set $K$ a factor system $\widehat{f}: \widehat{\Dc}_K\rightarrow\widehat{\Dc}_K$, where $\widehat{\Dc}_K$ under the quotient topology is a cactoid.

\bibliographystyle{plain}

\end{document}